\documentclass[12pt]{amsart}
\usepackage{latexsym,amssymb,amsmath}
\usepackage{amsmath,amsfonts}
\usepackage{comment}
\usepackage{enumerate}
\usepackage{todonotes}
\usepackage{cleveref}
\usepackage{epsfig}
\usepackage{graphicx}
\usepackage{verbatim,xcolor}
\newtheorem{theorem}{Theorem}[section]
\newtheorem{lemma}[theorem]{Lemma}

\newtheorem{definition}[theorem]{Definition}
\newtheorem{rmk}[theorem]{Remark}

\newtheorem{hyp}{Hypothesis}

\usepackage[letterpaper, margin=1in]{geometry}

\allowdisplaybreaks

\newcommand{\Hmm}[1]{\leavevmode{\marginpar{\tiny%
$\hbox to 0mm{\hspace*{-0.5mm}$\leftarrow$\hss}%
\vcenter{\vrule depth 0.1mm height 0.1mm width \the\marginparwidth}%
\hbox to 0mm{\hss$\rightarrow$\hspace*{-0.5mm}}$\\\relax\raggedright #1}}}
\newcommand{\nc}{\newcommand}
\nc{\les}{\lesssim}
\nc{\ges}{\gtrsim}
\nc{\nit}{\noindent}
\nc{\nn}{\nonumber}
\nc{\D}{\partial}
\nc{\diff}[2]{\frac{d #1}{d #2}}
\nc{\diffn}[3]{\frac{d^{#3} #1}{d {#2}^{#3}}}
\nc{\pdiff}[2]{\frac{\partial #1}{\partial #2}}
\nc{\pdiffn}[3]{\frac{\partial^{#3} #1}{\partial{#2}^{#3}}}
\nc{\abs}[1] {\lvert #1 \rvert}
\nc{\cAc}{{\cal A}_c}
\nc{\cE}{{\cal E}}
\nc{\mF}{{\mathcal F}}
\nc{\cF}{{\mathcal F}} 
\nc{\cP}{{\cal P}}
\nc{\cV}{{\cal V}}
\nc{\cQ}{{\cal Q}}
\nc{\cGin}{{\cal G}_{\rm in}}
\nc{\cGout}{{\cal G}_{\rm out}}
\nc{\cO}{{\cal O}}
\nc{\Lav}{{\cal L}_{\rm av}}
\nc{\cL}{{\mathcal L}}
\nc{\cH}{{\mathcal H}}
\nc{\cB}{{\cal B}}
\nc{\cZ}{{\cal Z}}
\nc{\cR}{{\cal R}}
\nc{\cT}{{\cal T}}
\nc{\cY}{{\cal Y}}
\nc{\cX}{{\cal X}}
\nc{\cXT}{{{\cal X}(T)}}
\nc{\cBT}{{{\cal B}(T)}}
\nc{\vD}{{\vec \mathcal{D}}}
\nc{\efield}{\mathcal{E}}
\nc{\vE}{{\vec \efield}}
\nc{\vB}{{\vec \mathcal{B}}}
\nc{\vH}{{\vec \mathcal{H}}}
\nc{\ty}{{\widetilde y}}
\nc{\tu}{{\widetilde u}}
\nc{\tV}{{\widetilde V}}
\nc{\Pc}{{\bf P_c}}
\nc{\bx}{{\bf x}}
\nc{\bX}{{\bf X}}
\nc{\by}{{\bf y}}
\nc{\bXYZ}{{\bf XYZ}}
\nc{\bY}{{\bf Y}}
\nc{\bF}{{\bf F}}
\nc{\bS}{{\bf S}}
\nc{\bz}{{\bf z}}
\nc{\br}{{\bf r}}
\nc{\dV}{{\delta V}}
\nc{\dE}{{\delta E}}
\nc{\TT}{{\Theta}}
\nc{\dPsi}{{\delta\Psi}}
\nc{\order}{{\cal O}}
\nc{\Rout}{R_{\rm out}}
\nc{\eplus}{e_+}
\nc{\eminus}{e_-}
\nc{\epm}{e_\pm}
\nc{\eps}{\varepsilon}
\nc{\vnabla}{{\vec\nabla}}
\nc{\G}{\Gamma}
\nc{\w}{\omega}
\nc{\mh}{h}
\nc{\mg}{g}
\nc{\sinc}{{\rm sinc}}
\nc{\vphi}{\varphi}
\nc{\tlambda}{\widetilde\lambda}

\nc{\g}{\gamma}
\nc{\ol}{\overline}

\newcommand\numberthis{\addtocounter{equation}{1}\tag{\theequation}}
  \nc{\F}{\mathcal F}
 
\def\R{\mathbb R}
\def\C{\mathbb C}
\nc{\T}{\mathbb T}
\nc{\Z}{\mathbb Z}
\nc{\N}{\mathbb N}
\nc{\pt}{\partial_t}
\nc{\la}{\langle}
\nc{\ra}{\rangle}
\nc{\infint}{\int_{-\infty}^{\infty}}
\nc{\halfwidth}{6.5cm}
\nc{\uu}{\" u}
\nc{\oo}{\" o}
\nc{\nlayers}{L} \nc{\nsectors}{M}
\nc{\indicator}{\mathbf{1}}
\nc{\Rhole}{R_{\rm hole}}
\nc{\Rring}{R_{\rm ring}}
\nc{\neff}{n_{\rm eff}}
\nc{\Frem}{F_{\rm rem}}
\nc{\DD}{\Delta}
\nc{\cD}{\mathcal D}
\nc{\rar}{\rightarrow}
\nc{\sgn}{{\rm sign}}
\nc{\Dx}{\mathcal{D}}
\nc{\non}{\nonumber}
\nc{\wh}{\widehat}
\nc{\minf}{m_{\infty}}
\nc{\im}{{\rm Im}}
\nc{\hatD}{\widehat{D}(k_1)}
\allowdisplaybreaks
\date{\today}

\title{Dispersive Estimates for Dirac Operators with General Domain Walls: One and Two Dimensions}

\author[]{M. Burak Erdo\u{g}an}
\address{Department of Mathematics, University of Illinois, Urbana, IL 61801, USA}
\email{berdogan@illinois.edu}

\author[]{Joseph Kraisler}
\address{Department of Mathematics, Williams College, Williamstown, MA 01267, USA}
\email{jk34@williams.edu}

\author[]{Amir Sagiv}
\address{Department of Mathematical Sciences, New Jersey Institute of Technology, University Heights, Newark, NJ 07102, USA}
\email{amir.sagiv@njit.edu}

\begin{document}
\begin{abstract}
We establish dispersive decay estimates for two-dimensional Dirac equations with bounded and unbounded domain walls, together with estimates for the analogous one-dimensional problem. These are paradigmatic models of bulk-edge correspondence in topological insulators, yet their dispersive dynamics have not been previously studied. We show that topologically equivalent models may exhibit qualitatively different dynamics. Our new approach allows us to analyze models which are beyond the usual framework of localized perturbations of constant-coefficient Dirac operators.
\end{abstract}

\maketitle


\section{Introduction}

Consider the dynamics of the two-dimensional Dirac equation with a domain wall,
\begin{subequations}\label{eq:intro-2d-dirac}
\begin{equation}
\begin{cases}
i\partial_t \alpha(t,x)=D\alpha(t,x)  \, ,\\
\alpha(0,x)=\alpha_0(x)\in L^2(\R^2;\C^2)  \, ,
\end{cases}
\end{equation}
where $x=(x_1, x_2)\in \R ^2$ and
\begin{equation}\label{eq:2d_Dirac}
D\equiv -i\sigma_1\partial_{x_1}-i\sigma_2\partial_{x_2}+m(x_2)\sigma_3 = \begin{pmatrix}
    m(x_2) & -i\partial_{x_1} -\partial_{x_2} \\
    -i\partial_{x_1} +\partial_{x_2} & -m(x_2)
\end{pmatrix}  \, .
\end{equation}
\end{subequations}
Here $\sigma_j$ are the Pauli matrices (see notations in Sec.\ \ref{sec:notations}), and the varying mass
term $m:\R \to \R$ models either a bounded domain wall, in which  $\lim_{x_2 \to \pm \infty} m(x_2)=\pm m_{\infty}$ with $0<m_{\infty}<\infty$; or an unbounded ($m_{\infty}=\infty$) domain wall.

Equation~\eqref{eq:intro-2d-dirac} is a paradigmatic PDE model for two-dimensional topological
insulators \cite{bal2019continuous, bal2024continuous}. The sign change in the  ``mass'' term $m$ between the lower and upper half-planes gives rise to
``edge modes'': energy transport only along the one-dimensional interface $\{x_2=0\}$. These edge modes are {\em topologically protected}, i.e., the net transport of energy along the edge is quantized and stable under perturbations of the Hamiltonian, all owed to topological considerations; see details in  Sec.\ \ref{sec:TI}.

The study of these edge modes, however, only provides a partial understanding of the dynamics of
\eqref{eq:intro-2d-dirac}. In this work, we study its {\em dispersive dynamics}: how waves
spread and decay under unitary evolution. Our results reveal the dependence of the dispersive dynamics on the domain wall, $m$. {\em In particular, we show that ``topologically equivalent'' models, which have the same net edge-transport, may exhibit qualitatively different dispersive dynamics, which are not accounted for by topological considerations.}

The distinction between bounded and unbounded domain walls is central. In the unbounded
case ($m_{\infty}=\infty$), the evolution decomposes into spectral branches which propagate along the edge
direction, $x_1$, while remaining confined in the transverse direction, $x_2$. While the
topologically protected edge-mode branch does not disperse, the remaining branches disperse at a $t^{-1/2}$ rate along the edge, similarly to waves under the flow of the massive one-dimensional Dirac
equation \cite{erdougan2021one}.

When the domain wall $m$ is bounded ($m_{\infty}<\infty$), the picture is richer. In addition to the edge-mode branches,
there is a genuine bulk component dispersing in both spatial directions. We show that this bulk
component satisfies a two-dimensional $t^{-1}$ global (unweighted) decay rate and, in the regular-threshold case, a faster local (weighted) decay rate up to $t^{-2}$. {\em Thus, different versions of
\eqref{eq:intro-2d-dirac} which have the same {\em edge}
index have very different dispersive behavior.} See Sec.\ \ref{sec:mainRes} for details.

The key challenge in analyzing the two-dimensional bounded-domain-wall \eqref{eq:intro-2d-dirac} is that {\em our Hamiltonian, $D$, is not a relatively compact perturbation of either the massive or
massless free Dirac Hamiltonian.} Therefore, dispersive decay estimates of \eqref{eq:intro-2d-dirac} do not fall directly within the existing framework of
\cite{burak2019limiting, erdougan2021massless, erdougan2017dirac, erdougan2018dispersive, herr2025strichartz}. 
Rather, after Fourier transform in the edge direction, $x_1$, the bounded-wall
operator decomposes into the family of one-dimensional Dirac operators:
\begin{equation}
\widehat{D}(k_1)=k_1\sigma_1-i\sigma_2\partial_{x_2}+m(x_2)\sigma_3 \, ,\qquad k_1\in\R \, .
\label{eq:intro-fiber-dirac}
\end{equation}
This naturally leads to the study of dispersion in one-dimensional Dirac Hamiltonians with a domain wall, which are of independent interest in the analysis of topological insulators and waves in periodic media \cite{drouot2021bulk, drouot2020, FLW-PNAS:14, FLW-MAMS:17, DomainWall2024}. These latter one-dimensional operators are also not localized perturbations of a
one-dimensional free Dirac operator, as in \cite{erdougan2021one}; nevertheless, they can be treated by a scattering
approach based on Dirac Jost solutions and the associated one-dimensional Schr{\"o}dinger
scattering theory \cite{deift1979inverse, goldberg2007dispersive, kraisler2024dispersive}.

\subsection{Motivation: Topological Insulators}\label{sec:TI}

The Dirac equation was first introduced as a model reconciling quantum mechanics and special relativity \cite{Thaller_Dirac}. Dirac operators also feature prominently in the study of waves in periodic media. For elliptic operators with periodic coefficients, such as Schr{\"o}dinger, Helmholtz, and Maxwell operators, and given certain special symmetries, such as those of a honeycomb lattice in graphene, the dispersion surfaces  cross conically at isolated points in quasi-momentum space. Such conical crossings are known as Dirac points, and wave packets which are spectrally localized near a Dirac point are governed, on the relevant time-scale, by an effective massless Dirac equation \cite{ammari2026wave, FW:14, sagiv2022effective}. Adding a constant mass term ($m\equiv {\rm const}$) opens a spectral gap at the Dirac point. {\em Varying} the mass term, in one space dimension, leads to edge-modes: proper eigenmodes of the Hamiltonian which are exponentially localized around the edge \cite{FLW-MAMS:17, FLW-PNAS:14, kraisler2024dispersive, drouot2020, drouot2021bulk}. In two space dimensions, by integrating together the one-dimensional Fourier fibers \eqref{eq:intro-fiber-dirac}, one obtains ``branches'' of edge modes, i.e., parts of the continuous spectrum wholly localized around the edge, but that consist of states which may travel and disperse {\em along} the edge.

Most relevant to this work is the role of the Dirac equation with a domain wall \eqref{eq:intro-2d-dirac} as a paradigmatic example for a two-dimensional PDE model of a {\em topological insulators}. The field of topological insulators was launched by the experimental observation of the integer quantum Hall effect in 1980 \cite{klitzing1980new, prange1990quantum}. Broadly speaking, it concerns the relation between physical observables and topological invariants associated with suitable classes of Hamiltonians. This topological perspective explains both the quantization of these observables and their stability under perturbations. Although we do not explicitly use the theory of topological insulators in our analysis, it provides the conceptual setting for our work. We therefore briefly review the relevant ideas; see the introductions \cite{asboth2016short, bernevig2013topological, moessner2021topological,  wen2017colloquium} for further detail.

In two space dimensions, gapped bulk Hamiltonians may carry topological
indices which are invariant under continuous perturbations that preserve
the bulk gaps. The observable consequences of this topology can be expressed
through the {\em bulk-edge correspondence} (BEC): an interface between
two gapped bulk Hamiltonians carries an edge-conductance index which,
loosely speaking, measures the net signed number of edge modes propagating
in one direction along the interface. The BEC then asserts that, when the
two half-space bulk media have different bulk indices, their difference
determines the edge index
\cite{hatsugai1993chern,schulz2000simultaneous}. Consequently, the
edge conductance is quantized, and the corresponding energy transport is
protected under continuous, potentially large, perturbations of the
Hamiltonian that preserve the bulk gap.

Continuum models (PDEs) have been important throughout the development of topological insulators and the BEC. Continuum PDE models have long played a role in this theory, from early works on magnetic Schr{\"o}dinger operators \cite{halperin1982quantized,frohlich2000extended, kellendonk2004quantization} to more recent studies of Schr{\"o}dinger operators in periodic media and their homogenized counterparts \cite{drouot2019bulk,drouot2021microlocal, shapiro2022tight}.

Among these models, the domain-wall Dirac operator
\eqref{eq:intro-2d-dirac} provides a particularly transparent effective
description of BEC. Operators of this form were studied in high-energy physics and topological
photonics
\cite{callan1985anomalies,fosco1999dirac,raghu2008analogs}.
The rigorous Fredholm-index framework for BEC in these models was more recently developed in
\cite{bal2019continuous,bal2023topological}. The change of sign of the mass, $m$, between the lower and upper half-planes,
produces a nonzero bulk-difference invariant. This bulk invariant equals the interface index
measuring the net asymmetric transport along the domain wall (for simplicity, the line $\{ x_2=0\}$ when $m$ is an odd function). Later works study more quantitative features of edge dynamics such as propagation of wave packets along curved interfaces \cite{bal2023edge}, in the presence of a magnetic field \cite{bal2024magnetic}, and under parametric
forcing \cite{bal2022multiscale}, as well as high-dimensional analogs \cite{bal2023topological}.

The BEC detects only net asymmetric propagation along the domain wall. Indeed, reducing the study of the family of models \eqref{eq:intro-2d-dirac} to integer indices inevitably leads to a loss of some interesting dynamical information. In particular, it leaves open the question of {\em dispersive}, infinite-time dynamics. This is the central contribution of this work. 

While there is a developed literature on dispersive and Strichartz estimates in the Dirac equation \cite{boussaid2006stable, carey2018global, d2005decay, burak2019limiting, erdougan2021massless,  erdougan2018dispersive, erdougan2019dispersive,  green2025massless, herr2025strichartz, kopylova2011weighted, kopylova2013dispersion,  kovavrik2022spectral, kraisler2024dispersive, pelinovsky2012asymptotic}, it treats only localized perturbations (short-range potentials)
of constant-coefficient (free) Dirac Hamiltonians. Domain-wall Hamiltonians do not fall under this framework, since they have different asymptotics on the two half-spaces. This is the main technical novelty of this work, and its consequences are discussed in Sec.\ \ref{sec:mainRes}.

\subsection{Results and Strategy}\label{sec:mainRes}

\begin{figure}[p]
\centering
\includegraphics[
    width=\textwidth,
    height=.78\textheight,
    keepaspectratio
]{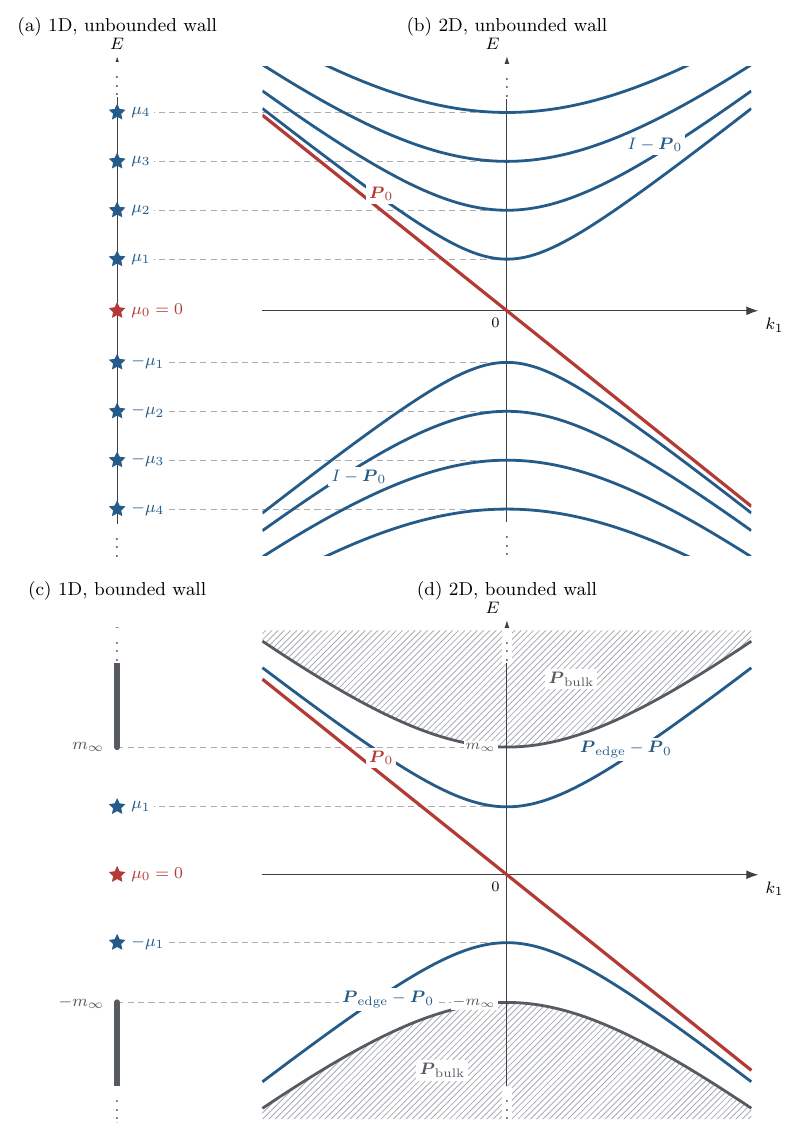}
\caption{Spectra for unbounded (a+b, top row) and
bounded (c+d, bottom row) domain walls. The left panels (a+c) show the
spectrum of the one-dimensional fiber operator $\widehat{D}$: stars denote point spectrum, solid curves continuous spectrum. Right panels (b+d) show the branches
they generate as $k_1$ varies for the two-dimensional $D$: solid curves denote
edge branches, and hatching denotes the two-dimensional bulk spectrum.}
\label{fig:spectrum}
\end{figure}
As discussed above, the two-dimensional Hamiltonian $D$ admits a Fourier transform decomposition in $x_1$ into a direct integral of one-dimensional fiber Hamiltonians, $\widehat D(k_1)$; see
\eqref{eq:intro-fiber-dirac}. Thus, the spectrum of $D$ is assembled from the spectra of
its fibers which, in turn, are determined by a supersymmetric (SUSY) pair of
one-dimensional Schr{\"o}dinger operators in the transverse variable, $x_2$,
schematically\footnote{This is the notation we use for the unbounded case (Sec.\ \ref{sec:unbd}). In the bounded case, we use a shifted version; see details in Sec.\ \ref{sec:2dbd_spec}.}
\begin{equation}\label{eq:Hpm-intro}
H_\pm=-\partial_{x_2}^2+m^2(x_2)\pm m'(x_2) \, .
\end{equation}
Determining the spectrum of $H_{\pm}$, as well as their scattering theory in the bounded
case, is a classical problem in SUSY quantum mechanics, with a known connection to the
Dirac equation; see, e.g., \cite{mielnik2004factorization, Thaller_Dirac}. Going back to the
fiber-Dirac operator, $\widehat{D}(k_1)$, nonzero spectral values of the Schr{\"o}dinger
operators, $\mu\in\sigma(H_\pm)$, generate Dirac energies of the form
$\lambda=\pm\sqrt{k_1^2+\mu}$. Note that, even if $\mu$ is a proper eigenvalue of
$H_{\pm}$, the corresponding branch of $\lambda$ lies in the {\em absolutely continuous} 
$L^2(\R^2)$ spectrum of $D$, because of the non-trivial dependence of $\lambda$ on $k_1$.

Even though spectrum of $D$ is purely absolutely continuous, and that $\sigma_{\rm ac}(D)=\R$, the spectrum of the fibers $\widehat{D}(k_1)$ informs a decomposition of $L^2(\R^2)$ into $D$-invariant subspaces, which feature
qualitatively different dispersive dynamics. These subspaces are illustrated in Fig.\ \ref{fig:spectrum}.\footnote{We defer their formal statements to the relevant sections.} First, the zero mode of $\widehat{D}(0)$ generates the zero
edge-mode branch, $\lambda(k_1)=-k_1$, whose topological properties are discussed above in
Sec.\ \ref{sec:TI}. Denoting the projection onto this branch by $P_0$, we have
that data in ${\rm image}(P_0)$ evolve as a traveling wave along the edge, without dispersing.
It is therefore the dynamics on the orthogonal subspace,
$ {\rm image}(I-P_0)$, in which we are interested in this
paper.

In the {\bf unbounded case}, the operators $H_\pm$ have pure point spectrum, with eigenvalues
$\{\mu_n\}_{n=0}^{\infty}$ tending to infinity.
Thus, the dynamics in ${\rm image}(I-P_0)$ decomposes into countably many
 edge-mode branches, with one-dimensional dispersion relations given by
$\pm\sqrt{\mu_n+k_1^2}$. 
\begin{theorem}[Informal formulation of Theorem~\ref{thm:unbndmdisp}]
Assume that $m$ is unbounded. Then,  
\begin{equation*}
\left\|e^{-itD}\langle D\rangle^{-\gamma}(I-P_0)\right\|_{L^1_{x_1}L^2_{x_2}\to
L^\infty_{x_1}L^2_{x_2}}
\lesssim
\langle t\rangle^{-1/2}
\, ,
\qquad
\gamma> 3/2 
\, .
\end{equation*}
\end{theorem}
To prove this bound, we expand the data in the transverse eigenbasis,
$\Gamma_n(x_2)$, and then analyze the one-dimensional $x_1$ dynamics on each branch by expressing the resolvent of $D$ using that of the
one-dimensional free Schr{\"o}dinger resolvent, similarly to the one-dimensional analysis of
\cite{erdougan2021one}. The unusual feature of our analysis is the use of the
 mixed norm, $L^\infty_{x_1}L^2_{x_2}$: it reflects the fact that energy only disperses along the edge ($x_1$), whereas the overall mass distribution across the edge-mode branches (the different $\Gamma_n$, see Fig.\ \ref{fig:spectrum}), is left unchanged by the dynamics.

{\bf In the bounded case,} the Schr{\"o}dinger operators $H_\pm$ (see \eqref{eq:Hpm-intro}) have absolutely continuous spectrum
$[m_\infty^2,\infty)$, and possibly finitely many discrete eigenvalues below
$m_\infty^2$. Thus, for each $k_1\in\R$,
\begin{equation*}
\sigma_{\rm ac}(\widehat D(k_1))
=
(-\infty,-\nu(k_1,0)]\cup[\nu(k_1,0),\infty)
\, ,
\qquad
\nu(k_1,k_2)\equiv\sqrt{k_1^2+k_2^2+m_\infty^2} \, .
\end{equation*}
Informally, the projection onto the bulk component is
$P_{\rm bulk}=\int^{\oplus}P_{\rm ac}(\widehat D(k_1))\,dk_1$.
The fiber point spectra generate the edge projection, $P_{\rm edge}$,
which contains the zero-branch projection, $P_0$. Thus,
\begin{equation*}
I_{L^2(\R^2;\C^2)}=P_{\rm edge}+P_{\rm bulk} \, .
\end{equation*}
The dynamics of $P_0$ and $P_{\rm edge}-P_0$ are essentially as
in the unbounded domain-wall case above.  We thus focus on the bulk component, which disperses in two dimensions, with dispersion relation given by
$\nu(k_1,k_2)$. We show:
\begin{theorem}[Informal formulation of the bulk estimate in Theorem~\ref{thm:2d_main_theorem}]
For $m_{\infty}<\infty$ 
\begin{equation*}
\left\|e^{-itD}\langle D\rangle^{-\gamma}P_{\rm bulk}\right\|_
{L^1(\R^2;\C^2)\to L^\infty(\R^2;\C^2)}
\lesssim
\langle t\rangle^{-1}
\, , \qquad \gamma > 2 \, .
\end{equation*}
\end{theorem}

This is the natural decay rate for the two-dimensional constant-mass Dirac equation \cite{erdougan2017dirac, erdougan2018dispersive}. The main
point is not the oscillatory integral itself, which is governed by the phase
$\nu(k_1,k_2)$, but rather the construction of a spectral representation for $P_{\rm bulk}$
whose amplitudes can be controlled uniformly in $(k_1,k_2)$.

This main difficulty already appears in one dimension. Indeed, omitting the subscript in~$x_2$, there is a fixed unitary matrix $U$ such that $U\widehat D(k_1)U^* =\Dx+k_1\sigma_2$, where $\Dx$ is (up to a fixed conjugation by Pauli matrices) the one-dimensional Dirac operator with a domain wall
\begin{equation}\label{eq:1d_dirac}
\Dx=-i\sigma_3\partial_x+m(x)\sigma_1 \, ,
\end{equation}
where $m(x)\to\pm m_\infty$ as $x\to\pm\infty$. Thus, obtaining dispersive decay estimates for $\Dx$ paves the way for the analysis of the two-dimensional problem above. This model is of independent interest in the analysis of topologically protected edge modes in one dimension \cite{drouot2021bulk, drouot2020, FLW-PNAS:14, FLW-MAMS:17, watson2018wavepackets}. While dispersive decay estimates for a family of models containing $\Dx$, when $m(x)={\rm sgn}(x)$ is a step function, were obtained in \cite{kraisler2024dispersive}, the analysis of the general case is stated informally below:
\begin{theorem}[Informal formulation of Theorem~\ref{thm:1d_unweighted}]
Let $P_{\rm c}(\Dx)$ denote the projection onto the absolutely continuous spectrum of $\Dx$. Then, for any  $m$ sufficiently regular,
\begin{equation*}
\left\|e^{-it\Dx}\langle \Dx\rangle^{-\gamma}P_{\rm c}(\Dx)\right\|_
{L^1(\R;\C^2)\to L^\infty(\R;\C^2)}
\lesssim
\langle t\rangle^{-1/2}
\, , \qquad \gamma > 3/2 \, .
\end{equation*}
\end{theorem}

The operators $D$ and $\Dx$ are not localized
perturbations of constant-mass Dirac operators,  and therefore the relevant standard theory does not apply.
Nevertheless, the {\em rates} of decay  and smoothing associated with $\Dx$ and $D$ agree with those of the Dirac equation in one \cite{erdougan2021one} and two \cite{erdougan2017dirac} dimensions with constant mass, respectively. Intuitively, this reflects dispersion far from the interface, where the mass is nearly constant, locally.

To isolate the effect of the interface, we turn our attention to {\em weighted}
(local) estimates: these require stronger localization of the initial data and
measure the supremum of its evolution near the interface. We localize only
transversely to the interface: in $x$ for the one-dimensional operator, $\Dx$,
and in $x_2$ for the two-dimensional operator, $D$. There appears to be
no improvement from additionally localizing along the interface, in the $x_1$
direction. The resulting weighted estimates feature improved rates, but hinge
on the regular-threshold assumptions stated precisely in
Theorems~\ref{thm:1dWeighted} and~\ref{thm:2dweighted}.

  \begin{theorem}[Informal formulation of Theorem~\ref{thm:1dWeighted} in one dimension]
Assume that $\Dx$ has no threshold resonance. Then, for any $\gamma>3/2$,
\begin{equation*}
\left\|
\langle x\rangle^{-1}
e^{-it\Dx}\langle\Dx\rangle^{-\gamma}P_c(\Dx)
\langle y\rangle^{-1}
\right\|_{L^1(\R;\C^2)\to L^\infty(\R;\C^2)}
\lesssim
\langle t\rangle^{-3/2}
\, .
\end{equation*}
\end{theorem}

\begin{theorem}[Informal formulation of Theorem~\ref{thm:2dweighted} in two dimensions]
Assume that $D$ is regular. Then, for any $\gamma>3$,
\begin{equation*}
\left\|
\langle x_2\rangle^{-1}
e^{-itD}\langle D\rangle^{-\gamma}P_{\rm bulk}
\langle y_2\rangle^{-1}
\right\|_{L^1(\R^2;\C^2)\to L^\infty(\R^2;\C^2)}
\lesssim
\langle t\rangle^{-2}
\, .
\end{equation*}
\end{theorem}

The improvements in decay rates by weighting are not unprecedented, and they depend on the behavior of the resolvent near the threshold: for one-dimensional constant-mass
Dirac operators with a potential, regular thresholds yield the same improvement (from $t^{-1/2}$ to $t^{-3/2}$)
in weighted estimates
\cite{erdougan2021one}, as they do in the case of the two-dimensional massless Dirac equation \cite{erdougan2021massless}. In the massive Dirac equation in two dimensions with a regular threshold, weighting leads to a much more modest improvement, from $t^{-1}$ to $t^{-1}(\log t)^{-2}$ \cite{erdougan2018dispersive}; see
Remark~\ref{rmk:2dweighted}.

Such improvements of the decay rate by weighting are significant in at least
two ways. First, they show that a property of the domain wall which is
undetectable by topological considerations can alter the dynamics near the
interface. Second, the improved rates are {\em integrable} in time, i.e., they
are of order $\langle t\rangle^{-\sigma}$ for some $\sigma>1$, so that $
\int_0^\infty \langle t\rangle^{-\sigma}\,dt<\infty $.
This property is important in the analysis of such phenomena as the Fermi
Golden Rule under non-autonomous parametric forcing \cite{hameedi2022radiative, miller2000metastability, soffer1998nonautonomous}, as well as the
stability of waves in the presence of nonlinearity \cite{boussaid2006stable, boussaid2019nonlinear, boussaid2019spectral, d2005decay, escobedo1997semilinear, herr2026decay, machihara2005endpoint, pelinovsky2012asymptotic, pelinovsky2014orbital}.

We end our introduction by briefly discussing the proofs of the weighted and unweighted estimates of the one-dimensional operator,
$\Dx$; the analysis for the two-dimensional bulk estimates follows analogously through the Fourier fibers $\widehat{D}(k_1)$. To express the propagator, we study the scattering theory of $\Dx$: as noted above, squaring $\Dx$ yields the SUSY Schr{\"o}dinger pair
$H_\pm$; see \eqref{eq:Hpm-intro}. Since these are short-range perturbations of
$-\partial_x^2+m_\infty^2$, we are able to invoke their known scattering theory \cite{deift1979inverse} to construct the resolvent of $\Dx$ from the associated (Dirac) Jost
solutions and their Wronskian,
$W_{\Dx}(k)$. Thus we obtain the propagator as a product of a
Klein-Gordon phase, with known decay rates (Sec.\ \ref{sec:KG}), and a scattering amplitude. To get the ``free'' Klein-Gordon rate, using Young's inequality, it is enough to control this amplitude uniformly in the
Wiener algebra, following
\cite{goldberg2004dispersive}. The weighted improvement follows by an additional
integration by parts, enabled by a non-vanishing of the associated Dirac Wronskian, in the spirit of \cite{goldberg2007dispersive}. Finally, we believe that the regularity hypotheses on the domain wall $m$ in the bounded case (Theorems \ref{thm:1d_unweighted}, \ref{thm:1dWeighted}, \ref{thm:2d_main_theorem}, \ref{thm:2dweighted}) could likely be lowered 
in all cases; see \cite{Egorova_2016,hill2020dispersive}.

\subsection{Structure of the paper}\label{sec:structure}
Central notations and general facts are tabulated in Sec.\ \ref{sec:notations}, as well as key results in the analysis of oscillatory integrals and a brief review of the theory of Wiener algebras.  The analysis of the bounded domain wall cases is then carried out in Secs. \ref{sec:1d_shrodinger}--\ref{sec:2d_bounded_estimates}, in the following way: underlying our analysis in one and two dimensions is the scattering theory of one-dimensional Schr{\"o}dinger operators with short-range potentials, presented in Sec.\ \ref{sec:1d_shrodinger}. Having defined the Schr{\"o}dinger Hamiltonians associated with $\Dx^2$, we mostly  review of the construction and properties of the associated Schr{\"o}dinger Jost functions and Wronskians, in the spirit of \cite{deift1979inverse}.  These objects are then
used in Sec.~\ref{sec:1d_estimates} to construct the Dirac Jost solutions
and resolvent kernel for $\Dx$, and thereby to prove the unweighted and
weighted decay estimates, Theorems~\ref{thm:1d_unweighted}
and~\ref{thm:1dWeighted}, respectively.  Section~\ref{sec:2d_bounded_estimates} adapts
this construction uniformly to the fiber operators $\widehat D(k_1)$
and proves the corresponding two-dimensional dispersive estimates, Theorem \ref{thm:2d_main_theorem} and \ref{thm:2dweighted}. Having carried out the analysis of the bounded domain wall, the two-dimensional Dirac Hamiltonian with {\em unbounded} domain is presented in Sec.\ \ref{sec:unbd}, whose main results are the dispersive decay estimates in Theorem \ref{thm:unbndmdisp}. Finally,
the auxiliary Sec.~\ref{sec:KG} proves the standard Klein--Gordon kernel
bounds used throughout the paper.

\subsection{Acknowledgments}
MBE is partially supported by NSF Grant No. \ DMS-2154031. AS is partially supported by NSF Grant No.\ DMS-2508811. The authors would like to thank Jacob Shapiro for useful comments and discussions.

\section{Notation and convention}\label{sec:notations}
\begin{itemize}
    \item The Fourier transform $\mF$ and its inverse are defined
    $$\mF[f](k) = \widehat{f}(k)= \int_{\R^d} e^{-ik\cdot x}f(x)dx,\quad \mF^{-1}[g](x)=\int_{\R^d}e^{ik\cdot x}g(k)\frac{dk}{(2\pi)^{d}}.$$
    \item The weighted Lebesgue spaces $L^{p}_{\sigma}$ for $p\in [1,\infty]$ and $\sigma\in\R$ are the space of functions with norm $\|f\|_{L^p_{\sigma}}=\| \langle \cdot \rangle^{\sigma} f\|_{L^p}$.  
    \item For two functions $f$ and $g$, we write $f(x)\sim g(x)$ as $x\to x_0$ if $\lim_{x\to x_0} f(x)/g(x) = 1$. We write $f(x)\lesssim g(x)$ if there is an $x$-independent constant such that $f(x)\leq C g(x)$.
    \item The Pauli matrices $\sigma_j$ for $j=1,2,3$ are defined
    \begin{equation}\label{eq:pauli}
        \sigma_1 = \begin{pmatrix}
            0 & 1\\
            1 & 0
        \end{pmatrix},\quad  \sigma_2 = \begin{pmatrix}
            0 & -i\\
            i & 0
        \end{pmatrix},\quad  \sigma_3 = \begin{pmatrix}
            1 & 0\\
            0 & -1
        \end{pmatrix}.
    \end{equation}
    \item {\bf The Wiener algebra,} $A(\R)$, and unital Wiener algebra, $A_1(\R)$, are defined
\begin{align*}
   A(\R) = \{ f \ |\  \widehat{f}\in L^1(\R)\}, \qquad A_1(\R) = \{ g+c\ ~|~  c\in\C ~ ,~~g\in A(\R)\} \, .
\end{align*}
We define the associated norms
$$ \|f\|_{A (\R)} \equiv \|\widehat f \|_{L^1},\quad  \|f\|_{A_1(\R)} \equiv \| \widehat{f}\|_{TV}= |c| + \|\widehat{g}\|_{L^1}.$$

We will also work with the larger space $\mF TV(\R)\supset A_1(\R)\supset A(\R)$, the algebra of continuous functions whose Fourier transform is a finite Borel measure, with the norm
$$
\|f\|_{\mF TV(\R)}\equiv\|\widehat{f}\|_{TV}.
$$
Note that $A(\R)$ is an ideal in $\mF TV(\R)$ and $\|fg\|_{A(\R)}\leq \|f\|_{\mF TV(\R)} \|g\|_{A(\R)}.$
The following lemma is Lemma 6.3 in \cite{Katznelson_2004}.
\begin{lemma}[Wiener]\label{lem:Wiener}
Let $f_1,f_2\in A(\R)$ such that $f_1$ has compact support and $f_2$ is nonzero on the support of $f_1$, then there exists $g\in A(\R)$ such that $f_1 = gf_2$.
\end{lemma}
We also state a version for the space $A_1(\R)$.
\begin{lemma}\label{lem:WienerA1}
Suppose $f=g+c\in A_1(\R)$ with $g\in A(\R)$ and $c\neq 0$, such that $f$ never vanishes. Then $1/f\in A_1(\R)$. 
\end{lemma}
Finally, we have the following useful embedding: $H^s(\R) \hookrightarrow A(\R)$ for all $s>1/2$.

    \item We use the following version of the Van der Corput Lemma throughout the paper.

    \begin{lemma}[Van der Corput \cite{stein1993harmonic}]\label{lem:vdc}
    Let $\phi:\R \to \R$ be  smooth and let $\psi \in C^{\infty}_c (\R )$. If $|\partial_k^{2}\phi(k)|\geq \lambda > 0$ for all $k\in {\rm supp}(\psi)$, then there exists a universal constant $C>0$ such that
    \begin{equation*}
        \Big| \int\limits_{\R} 
        e^{i\phi (k)}\psi(k) \, dk \Big| \leq C \lambda^{-\frac12} \left\| \partial_k \psi(k) \right\|_{L^1{(\R_k)}} \, .
    \end{equation*}
\end{lemma}

\end{itemize}

\section{Review of Scattering theory for one-dimensional Schr\"odinger operators}\label{sec:1d_shrodinger}
In this section we review the scattering theory for one-dimensional  Schr\"odinger Hamiltonians with decaying potentials. Many of the results can be found in \cite{deift1979inverse, goldberg2007dispersive, goldberg2004dispersive, Thaller_Dirac}, but we include the discussion for completeness. These results, and in particular Lemmas \ref{lem:wronratio} and \ref{lem:badside}, are  fundamental to studying the Dirac operators in Secs.\ \ref{sec:1d_estimates} and \ref{sec:2d_bounded_estimates}.

\subsection{Schr\"odinger Jost functions}\label{sec:Schrodinger_Jost}
Consider a one-dimensional Schr\"odinger operator with a real-valued potential:  
\begin{align*}
    H = -\partial_x^2 + V(x) \, ,\qquad V(x)\in L^1_1(\R) \, ,
\end{align*}
 where, for $\sigma\in\R$,
\begin{align*}
    \|V\|_{L_\sigma^1(\R)} \equiv \int_{\R}\langle x\rangle^\sigma \vert V(x)\vert\ dx \, .
\end{align*}
For each $k\in\R$, $k\neq 0$, one can define the corresponding \textbf{Jost functions} $f_{\pm}(x,k)$ as the unique functions satisfying
\begin{align}\label{eq:HJost}
    (H - k^2)f_{\pm} = 0 \, ,\qquad f_{\pm}(x,k)&\sim e^{\pm ikx}\, ,\quad  {\rm as}~~ x\to \pm \infty \,.
\end{align}  
To each Hamiltonian $H$ we associate two special Wronskians $W_{H}(k)$ and $\widetilde{W_{H}}(k)$ defined
\begin{align*}
    W_{H}(k) \equiv  W[f_{+}(\cdot,k),f_{-}(\cdot,k)] \, , \qquad \widetilde{W_{H}}(k) \equiv W[f_{-}(\cdot,k),f_{+}(\cdot,-k)]
\end{align*}
with the convention that the Wronskian of two functions $f(x), g(x)$ is $W[f,g] = f(x)g'(x)-f'(x)g(x)$. As $f_{\pm}(x,-k)$ satisfy the same ordinary differential equation as $f_{\pm}(x,k)$, and since the Wronskians $W_H, \widetilde{W}_H$ are independent of $x$, and by the asymptotics of $f_{\pm}$, a standard derivation yields the following relations
\begin{align}\label{eq:fpm_relation}
    f_{-}(x,k) = \alpha(k)f_{+}(x,k) + \beta(k)f_{+}(x,-k) \, ,
\end{align}
where $\alpha(k) = \dfrac{\widetilde{W_{H}}(k)}{-2ik}$ and $\beta(k)=\dfrac{W_{H}(k)}{-2ik}$. There is an analogous identity for $f_{+}(x,k)$ in terms of $f_{-}(x,\pm k)$.

Closely related are the modified Schr\"odinger Jost functions $\varphi_{\pm}(x,k)$, where we cancel the asymptotic phase behavior by
\begin{align*}
    e^{+ikx}\varphi_{+}(x,k) = f_{+}(x,k)\, ,\qquad e^{-ikx}\varphi_{-}(x,k) = f_- (x,k)  \, .
\end{align*}
We summarize properties of the modified Schr\"odinger Jost functions below. By Deift and Trubowitz \cite{deift1979inverse}, we have the following bounds on the Fourier Transform of the modified Jost functions, $\varphi_{\pm}$, and their derivatives. Let $\widehat{\varphi}_{\pm}(x,\xi)$  be the the $k\to\xi$ Fourier transforms of these functions, and let $\delta_0$ be the Dirac measure at $\xi=0$, then 
\begin{subequations}\label{eq:DT_bounds}
    \begin{align}
    &\sup_{x\geq 0}|\widehat{\varphi} _{+}(x;\xi)-\delta_0 |\lesssim I(\xi) \,, \qquad {\rm where}~~~ I(\xi) \equiv  \int_{|\xi|}^{\infty} |V(t)| \,  dt  \, ,\\
    &\sup_{x\leq 0}|\widehat{\varphi} _{-}(x;\xi)-\delta_0|\lesssim I(\xi) ,\\
    &|\partial_{x}\widehat{\varphi} _{+}(x;\xi)|\lesssim I(\xi)+|V(x+\xi)|,\qquad x\geq 0 \, , \\
    &|\partial_{x}\widehat{\varphi} _{-}(x;\xi)|\lesssim I(\xi)+|V(x-\xi)|,\qquad x\leq 0 \, .
\end{align}
\end{subequations}
In particular if $V\in L^1_1(\R)$ then we have that $\varphi_{+}(x;k)\in A_1 (\R)$ for all $x\geq 0$  and $\varphi_{-}(x,k)\in A_1 (\R)$  for all $x\leq 0$. Since the bounds $I(\xi)$ do not depend on $x$, the $A_1(\R)$ norm of both functions in the corresponding half-line are uniformly bounded in $x$. The following lemma is similar in spirit to \cite[Lemma 5]{goldberg2004dispersive} and \cite[Proposition 6]{goldberg2007dispersive}, but allows us to handle the Wronskians  globally in $k$.

The next two lemmas tabulate the properties of the $k$-Fourier transform of the Wronskian (Lemma \ref{lem:wronratio}) and of products of the Wronskian and the modified Jost functions in the non-resonant case (Lemma \ref{lem:badside}). These Lemmas will be used in the proofs of the unweighted and weighted dispersive estimates, respectively.

\begin{lemma}\label{lem:wronratio}
Let $\lambda(k)$ be a complex valued $C^\infty$ function on $\R$ satisfying   $\lambda(k)=k$ for $|k|>1$, and $|\lambda(k)|\gtrsim 1$ on $\R$. If $W_H(0)\neq 0$, then 
 $$
\widetilde{W}_H(k)~, ~\frac{\widetilde{W}_H(k)}{\lambda(k)}\in A(\R) \, ,\qquad \frac{\lambda(k)}{W_{H}(k)}\in A_1(\R)\,,
$$
provided that   $V\in L_1^1$. 
If instead $W_H(0)=0$, then the same conclusions hold with $\lambda(k)=k$, provided that $V\in L_2^1$. 

Finally, in both cases, we have 
$$\frac{\partial_k W_H(k)}{\lambda(k)} \in A_1(\R)\,,
$$
provided that $V\in L_2^1$.
\end{lemma}

\begin{proof}
In terms of the modified Schr\"odinger Jost functions 
\begin{align}\label{eq:WHform}
    W_{H}(k) = \varphi_{+}(0,k)\partial_x\varphi_{-}(0,k)-\varphi_-(0,k)\partial_x \varphi_{+}(0,k)-2ik\varphi_{+}(0,k)\varphi_{-}(0,k) .
\end{align}
We start with the case that $W_{H}(0)\neq 0$. Recall that $H^1 (\R)\hookrightarrow A(\R)$. Since $\frac{k}{\lambda(k)}-1\in H^1$ and $\frac1{\lambda(k)}\in H^1$, and since the $1$ term under the Fourier transform contributes a $\delta_0$ term, then  $\frac{k}{\lambda(k)}, \frac1{\lambda(k)}\in A_1(\R)$. Also using that $\varphi_\pm(0,k)$ and $\partial_x\varphi_\pm(0,k)$ are in $A_1(\R)$ (by \eqref{eq:DT_bounds}), we conclude that $W_{H}(k)\lambda^{-1}(k) \in A_1(\R)$. In addition, by the regularity assumption and the properties of $\lambda(k)$, we see that $W_{H}(k)\lambda^{-1}(k)  $ is nonvanishing. Therefore, by Wiener's lemma (Lemma \ref{lem:WienerA1}), we have 
\begin{align*}
\frac{\lambda(k)}{W_{H}(k)}\in A_1(\R).
\end{align*}
As for the case that $W_{H}(0)=0$, since this is a simple zero of $W_{H}$ \cite[Theorem 1, V (ii)] {deift1979inverse}, we have chosen $\lambda(k)=k$ in the resonant case, so we can get that $W_{H}(k)\lambda^{-1}(k)$ is continuous and nonzero. Let $\chi(k)$ be a smooth compactly supported function which is $1$ in a neighborhood of the origin. We have
\begin{equation}
    \frac{W_{H}(k)}{\lambda(k)} = \chi(k)\frac{W_{H}(k)}{\lambda(k)} + (1-\chi(k))\frac{W_{H}(k)}{\lambda(k)} \, .
\end{equation}\label{eq:WsOverLambda_decomp}
The second term is in $A_1(\R)$ by the same argument as above. For the first term, first rewrite
\begin{align*}
    \mF\left[\frac{\chi(k)W_{H}(k)}{k}\right](\xi) = i\int_{-\infty}^{\xi} \mF[\chi(k)W_{H}(k)](\eta)d\eta \, .
\end{align*}
Since $\int_{\R} \mF [g](\eta) \, d\eta = g(0)$, and since we know that $\chi (0)W_H(0)=0$, we can change the domain of integration to yield
\begin{align}\label{eq:FT_divk}
    \mF\left[\frac{\chi(k)W_{H}(k)}{k}\right](\xi) = -i\int_{\xi}^{\infty}\mF[\chi(k)W_{H}(k)](\eta)d\eta \, .
\end{align}
 Due to the strengthened hypothesis on the potential, $\chi(k)W_{H}(k)$ has Fourier transform in $L^1_1(\R)$ \cite[Lemma 5]{goldberg2004dispersive}, and so the {\em integral} in \eqref{eq:FT_divk} is in $L^1 (\R)$. Hence, $\chi(k)W_{H}(k)/\lambda(k)\in A(\R)$. To conclude this argument, since we have shown that both terms on the right hand side of \eqref{eq:WsOverLambda_decomp} are in $A_1 (\R)$, we now have that $W_{H}(k)/\lambda(k)\in A_1(\R)$ as well. Since we already showed it is  non-vanishing, we can once again use Wiener's lemma (Lemma \ref{lem:WienerA1}) to yield $\frac{\lambda(k)}{W_{H}(k)}\in A_1(\R)$. 
 
 For $\widetilde{W}_H(k)$, note that
\begin{align*}
    \widetilde{W_{H}}(k) = \varphi_-(0,k)\partial_x \varphi_+(0,-k)-\varphi_+(0,-k)\partial_x \varphi_-(0,k)\in A(\R) \, .
\end{align*}
Moreover, in the case that $W_{H}(0)=0$, we also have $\widetilde{W}_H(0)=-W_H(0)=0$ and hence
\begin{align*}
    \mF\left[\frac{\widetilde{W}_H(k)}{k}\right](\xi) =  i\int_{-\infty}^{\xi} \mF[\widetilde{W}_{H}(k)](\eta)d\eta = - i\int_{\xi}^{\infty}\mF\left[\widetilde{W}_H(k)\right](\eta)d\eta,
\end{align*}
from which $\frac{\widetilde{W}_H(k)}{k}\in A(\R)$ provided that $V\in L^1_2$; see \cite{goldberg2004dispersive} for more details. 

The final claim for the derivative of $W_H$ follows similarly from \eqref{eq:WHform} noting from \eqref{eq:DT_bounds} that $|\mF(\partial_k \varphi_{\pm}(0;k))(\xi)|\les |\xi|I(\xi)$ and $|\mF(\partial_k\partial_x \varphi_{\pm}(0;k))(\xi)|\les |\xi|(I(\xi)+|V(\xi)|)$.
\end{proof}

By definition, each Jost function has known asymptotics on one side of $\R$. On the other side (e.g., $x\to -\infty$ for $f_+$), the asymptotics are not a priori prescribed. Nevertheless, by the decomposition \eqref{eq:fpm_relation}, we obtain Wiener-algebra norm bounds, potentially non-uniform in~$x$.
 
\begin{lemma}\label{lem:badside} Assume that  $k=0$ is a regular point of the spectrum of $H$, i.e., $W_H(0)\neq 0$.  Let $\psi_\pm(x,k)\equiv\frac{k}{W_H(k)}\varphi_\pm(x,k)$. Then,  for $V\in L^1_1$ and  $\pm x<0$,
$$\|\psi_\pm(x,\cdot)\|_{\mF TV(\R)}\les  1\, .$$
For $V\in L^1_3$ and  $\pm x<0$ we have 
$$\|\varphi_\pm(x,\cdot)\|_{\mF TV(\R)} \les \la x\ra \, , \qquad  \
\|\partial_k\varphi_\pm(x,\cdot)\|_{\mF TV(\R)} \les \la x\ra^2 \, .
$$
In the resonant case, for $V\in L^1_2$ and  $\pm x<0$, we have 
$$\|\psi_\pm(x,\cdot)\|_{\mF TV(\R)}\les  1.$$
\end{lemma}
\begin{proof} WLOG we consider the $-$ case with $x>0$. Recall from \eqref{eq:fpm_relation} that we have 
$$ 
    f_{-}(x,k) = \alpha(k)f_{+}(x,k) + \beta(k)f_{+}(x,-k),
$$ 
where $\alpha(k) = \dfrac{\widetilde{W_{H}}(k)}{-2ik}$ and $\beta(k)=\dfrac{W_{H}(k)}{-2ik}$.
This implies
$$\psi_-(x,k)=\frac{ie^{2ikx}}2\frac{\widetilde{W}_H(k)}{W_{H}(k)}\varphi_{+}(x,k)+ \frac{i}2  \varphi_+(x,-k).$$
Since $\varphi_+ \in A_1 (\R)$ by \eqref{eq:DT_bounds}, then for the claims on $\psi_-$, it suffices to prove that 
$$\big\|\tfrac{\widetilde{W}_H(k)}{W_{H}(k)}\big\|_{A(\R)}=\big\|\tfrac{\widetilde{W}_H(k)}{\lambda(k)}\tfrac{\lambda(k)}{W_{H}(k)}\big\|_{A(\R)}\les 1 \, ,
$$ 
which follows immediately from Lemma ~\ref{lem:wronratio}.\\

Next, when $k=0$ is regular and $V\in L^1_3$, we write
\begin{multline*}\varphi_-(x,k)= \frac{ie^{2ikx}}2\frac{\widetilde{W}_H(k)}{k}\varphi_{+}(x,k)+ \frac{i}2  \frac{W_{H}(k)}{k} \varphi_+(x,-k) \\
=  \frac{i }2\frac{\widetilde{W}_H(k)+W_H(k)}{k}\varphi_{+}(x,-k) -\frac{i\widetilde W_{H}(k)}2 \  \frac{\varphi_+(x,-k)-\varphi_+(x,0)}{k} \\ + \frac{i \widetilde W_H(k)}2 \  \frac{ \varphi_{+}(x,k)-\varphi_+(x,0)}{k}  + \frac{i \widetilde W_H(k)}2  \varphi_{+}(x,k) \  \frac{e^{2ikx}-1}{k} .
\end{multline*}
The Fourier transform of $\dfrac{e^{2ikx}-1}{k}$ is a constant multiple of the indicator function on $[0,2x]$, which has $L^1$ norm proportional to $|x|$ and the Fourier transform of its $k$-derivative has $L^1$ norm proportional to $x^2$. Using arguments similar to those in Lemma~\ref{lem:wronratio}, each other term (and their $k$-derivative) has a Fourier transform with $L^1$ norm which is uniformly bounded in $x>0$. We require that $V\in L^1_3$ to handle the terms   $\partial_k\left(\frac{\widetilde{W}_H(k)+W_H(k)}{k}\right)$ and $\partial_k\left(\frac{ \varphi_{+}(x,\pm k)-\varphi_+(x,0)}{k}\right)$ using the identity \eqref{eq:FT_divk} and \cite[Lemma 5]{goldberg2004dispersive}.
\end{proof}

\subsection{Supersymmetric (SUSY) Hamiltonians}\label{sec:SUSY}
Here we briefly recall the theory of SUSY Schr{\"o}dinger operators, and introduce two such operators, $H$ and $\widetilde{H}$, from which in Secs.\ \ref{sec:1d_dirac_jost} and \ref{sec:1d_dirac_W}, we construct the scattering theory of the Dirac operator $\Dx$.

Suppose that $m\in W^{1,\infty}(\R)$ and define the pair of operators $A$ and $A^*$
\begin{align}\label{def:AA*}
    A = \partial_x + m(x) \, ,\qquad A^{*} = -\partial_x + m(x) \, .
\end{align}
Composing these operators in each order leads to the identity
\begin{align*}
    AA^{*} = H + \minf^2 \, ,\qquad A^{*}A = \widetilde{H}+\minf^2 \, .
\end{align*}
where $H,\widetilde{H}$ are a pair of Schr\"odinger operators 
\begin{subequations}\label{eq:1d_Schrodinger}
\begin{align}
    H &\equiv -\partial_x^2 + V (x)  \, , \qquad \widetilde{H}\equiv -\partial_x^2 +\widetilde V (x) \, ,\qquad  x\in\R\, , \qquad {\rm where}\\
    V (x)&\equiv  m^2(x)-\minf^2+ m'(x) \, ,\qquad   \widetilde V (x) \equiv  m^2(x)-\minf^2- m'(x) \,. 
\end{align}
\end{subequations}
The operators $H$, $\widetilde{H}$ are known as a supersymmetric pair. We collect some well-known facts about these operators, whose proofs can be found in ~\cite{ReedSimonIV, Thaller_Dirac}. 
\begin{lemma}\label{lem:SUSY_prop}
 Suppose $m^2(x)-\minf^2, m'(x)\in L^1_1(\R)$. Then we have the following properties of the spectra of the supersymmetric pair of Hamiltonians $H$ and $\widetilde{H}$.
\begin{enumerate}[i)]
    \item $\sigma(H)=\sigma(\widetilde{H})\setminus\{-\minf^2\}$.
    \item $-\minf^2\in\sigma(\widetilde{H})$ if and only if $\phi_0(x)=e^{-\int_0^{x} m(z)dz}\in L^2(\R)$. In this case 
    $$\widetilde{H}\phi_0 = -\minf^2\phi_0 \,.$$
    \item The essential spectrum of $H$ is $\sigma_{\rm ess}(H)=[0,\infty)$.
    \item $H$ has finitely many simple eigenvalues in the interval $(-\minf^2,0)$.  
\end{enumerate}
\end{lemma}
\begin{rmk}
Note that we define $H$, $\widetilde{H}$ such that the potentials $V$, $\widetilde{V}$, are asymptotically zero, as opposed to the usual definitions $AA^*$ and $A^*A$ which are non-negative operators. 
\end{rmk}
Let $f_{\pm}(x,k)$ and $\widetilde{f}_{\pm}(x,k)$, for $k\neq 0$, be the Jost functions defined in \Cref{eq:HJost} for the Hamiltonians $H$ and $\widetilde{H}$, respectively. We have the intertwining relations
\begin{subequations}\label{eq:intertwining}
\begin{align}
    &A\widetilde{f}_{+} = (ik+\minf)f_{+} \, ,\qquad A^{*}f_{+} = (-ik+\minf)\widetilde{f}_{+}  \, ,\\
    &A\widetilde{f}_{-} = -(ik+\minf)f_{-} \, ,\qquad A^{*}f_{-} = -(-ik+\minf)\widetilde{f}_{-}  \, .
\end{align} 
\end{subequations}
The Wronskians determined by the Jost functions associated to each operator $H$ and $\widetilde{H}$ can also be related to one another.
\begin{lemma}\label{lem:Wronskian_identities}
We have the following relations between Wronskians for $H$ and $\widetilde H$: 
\begin{align*}
     W_{H}(k) =\frac{ik-\minf}{ik+\minf}W_{\widetilde{H}}(k) \, ,\qquad 
      \widetilde{W_{H}}(k)=-\widetilde{W_{\widetilde{H}}}(k) \,  .
\end{align*}
\end{lemma}
\begin{proof}
First we state an identity related to the Wronskian. For any two functions $g$ and $h$,
\begin{align*}
    W[Ah,Ag] &= (A^{*}A g)(Ah) - (Ag)(A^{*}Ah) \, .
\end{align*}
If $g$ and $h$ satisfy $A^{*}Ag = (k^2+\minf^2) g$ (and the same for $h$), as do the functions $\widetilde{f}_{\pm}(x, k)$ and $\widetilde{f}_{\pm}(x, -k)$), then
\begin{align*}
    W[Ag,Ah] &= (k^2+\minf^2)W[g,h] \, .
\end{align*}
Thus
\begin{align*}
    W[f_{+}(\cdot,k),f_{-}(\cdot,k)] &= \frac{1}{-(ik+\minf)^2}W[A\widetilde{f}_{+}(\cdot,k),A\widetilde{f}_{-}(\cdot,k)]\\
    &=\frac{(k^2+\minf^2)}{-(ik+\minf)^2}W_{\widetilde{H}}(k) = \frac{ik-\minf}{ik+\minf}W_{\widetilde{H}}(k) , 
\end{align*}
and analogously we derive the identity for $\widetilde{W_{H}}(k)$.
 
\end{proof}

\section{Dispersive Estimates for one-dimensional Dirac operators with bounded domain wall}\label{sec:1d_estimates}
In this section we study the one-dimensional Dirac operator $\Dx$, given in \eqref{eq:1d_dirac}. As noted in the introduction, the techniques developed here are useful to approach the two-dimensional settings (Sec.\ \ref{sec:2d_bounded_estimates}). At the same time, estimates on the operator $\Dx$ are interesting in their own rights, and indeed were studied for the specialized case of  $m(x)={\rm sgn}(x)$ in \cite{kraisler2024dispersive}.

We begin by developing the scattering theory associated to the operator in Secs.\ \ref{sec:1d_dirac_jost} and \ref{sec:1d_dirac_W}. This theory is primarily built on top of the scattering theory of the SUSY Hamiltonians $H$ and $\widetilde{H}$ in Sec.\ \ref{sec:SUSY}. Using the Dirac scattering theory, in Sec.\ \ref{sec:1d_resolvent} we derive an expression for the continuous part of the spectral measure and the propagator. Finally, the proofs of dispersive estimates, Theorems \ref{thm:1d_unweighted} and \ref{thm:1dWeighted}, are provided in Sec.\ \ref{sec:1d_proofs}.

We consider the Dirac operator in the presence of a domain wall 
\begin{align*}
    \Dx = -i\sigma_3\partial_x + m(x)\sigma_1,
\end{align*}
where $m(x)\in C^{1}(\R)$ satisfies the asymptotics $
    \lim_{x\to\pm\infty}m(x)=\pm \minf$.

\begin{hyp}\label{hyp:domain_wall}
We assume that the functions $m^2(x)-\minf^2, m'(x)\in L^1_1(\R)$.
\end{hyp}
We cite some facts about the spectrum of the operator \eqref{eq:1d_dirac} found in \cite[Theorem 2.1]{Watson_2020}.
\begin{theorem}
The $L^2(\R)$ spectrum of the operator $\Dx$ comprises three pieces.
\begin{enumerate}[i)]
    \item Essential spectrum $(-\infty,-\minf]\cup[\minf,\infty)$.
    \item A simple eigenvalue at $0$ with corresponding eigenfunction 
    \begin{align*}
    \alpha_0(x) = \frac{1}{\sqrt{2}} \binom{1}{i} \exp{\left(-\int_0^x m(s)ds\right)}\ .
\end{align*}
    \item A finite set of additional simple eigenvalues in the gap $(-\minf,\minf)$.
\end{enumerate}
\end{theorem}

\subsection{One-dimensional Dirac Jost functions}\label{sec:1d_dirac_jost}
For each $k\in\R$, $k\neq 0$, let $f_{\pm}(x,k)$ and $\widetilde{f}_{\pm}(x,k)$ be the Jost functions associated to $H$ and $\widetilde{H}$, respectively. Then the Dirac Jost functions $g_{\pm}(x,k)$ are defined
\begin{align}\label{eq:1d_gpm_def}
    g_{\pm}(x,k) \equiv  \frac{1}{2}(\Dx+\nu(k) )S\begin{pmatrix}
        f_{\pm}(x,k) \\ \widetilde{f}_{\pm}(x,k)
    \end{pmatrix}  \, , \qquad {\rm where}~~~ S \equiv  \frac{1}{\sqrt{2}}\begin{pmatrix}
    1 & 1\\
    i & -i
\end{pmatrix} \, ,
\end{align}
and
\begin{align*}
    \nu(k) =\sqrt{k^2+\minf^2}.
\end{align*}
We show that, indeed, $g_{\pm}$ are the Jost solutions of the Dirac Hamiltonian $\Dx$, in the sense that they satisfy analogous conditions to the Schr{\"o}dinger Jost solutions
\begin{theorem}
The functions $g_{\pm}(x,k)$ are the Jost functions for the positive part of the spectrum of $\Dx$ in the sense that
\begin{align}\label{eq:DJost_evalue}
    (\Dx-\nu(k) )g_{\pm} = 0,
\end{align}
along with the asymptotics
\begin{align}\label{eq:DJost_asympt}
    g_{\pm}(x,k) &\sim \frac{1}{\sqrt{2}}\begin{pmatrix}
        \nu(k)\pm k\\
        \pm\minf
    \end{pmatrix}e^{\pm ikx},\qquad x\to \pm\infty.
\end{align}
\end{theorem}

\begin{rmk}
One can similarly define Jost functions, $h_{\pm}(x,k)$, for $k\neq 0$, associated to the negative part of the spectrum of $\Dx$ by the formula
\begin{align*}
    h_{\pm}(x,k) = \frac{1}{2}(\Dx-\nu(k))S\begin{pmatrix}
        f_{\pm}(x,k) \\ \widetilde{f}_{\pm}(x,k)
    \end{pmatrix} .
\end{align*}
\end{rmk}
\begin{proof}
Using the matrix $S$ defined above, we have
\begin{align}\label{eq:SDS}
    S^*\Dx S = \begin{pmatrix}
        0 & -iA\\
        iA^{*} & 0
    \end{pmatrix},\qquad S^* \Dx ^2 S &=\begin{pmatrix}
        AA^{*} & 0\\
        0 & A^{*}A
    \end{pmatrix}.
\end{align}
From this, it is straightforward to see that $g_{\pm}$ solve the necessary equation as
\begin{align*}
    S^*(\Dx-\nu)g_{\pm} &= \frac{1}{2}S^*(\Dx ^2 - \nu(k) ^2) S\begin{pmatrix}
        f_{\pm}(x,k)\\
        \widetilde{f}_{\pm}(x,k)
    \end{pmatrix}\\
    &= \frac{1}{2}\begin{pmatrix}
        H-k^2 & 0\\
        0 & \widetilde{H}-k^2
    \end{pmatrix}\begin{pmatrix}
        f_{\pm}(x,k)\\
        \widetilde{f}_{\pm}(x,k)
    \end{pmatrix}= 0.
\end{align*}
Moreover $S^*$ is invertible and so, necessarily, $(\Dx-\nu)g_{\pm}=0$. Now we prove the asymptotics. By the identity for $S^* \Dx S$, \eqref{eq:SDS}, we have
\begin{align*}
    (\Dx+\nu(k) )S &= S\begin{pmatrix}
        \nu(k)  & -iA \\
        iA^{*} & \nu(k) 
    \end{pmatrix} ,
\end{align*}
from which the formula for $g_{\pm}(x,k)$, \eqref{eq:1d_gpm_def}, can be recast as
\begin{align*}
    g_{\pm}(x,k) &= \frac{1}{2}(\Dx+\nu(k) )S\begin{pmatrix}
        \widetilde{f}_{\pm}(x,k)\\ f_{\pm}(x,k)
    \end{pmatrix}\\
    &= \frac{1}{2}S\begin{pmatrix}
        \nu(k)  & -iA \\
        iA^{*} & \nu(k) 
    \end{pmatrix}\begin{pmatrix}
        f_{\pm}(x,k)\\ \widetilde{f}_{\pm}(x,k)
    \end{pmatrix}\\
    &=\frac{1}{2}S\begin{pmatrix}
        \nu(k) f_{\pm}(x,k) -iA\widetilde{f}_{\pm}(x,k)\\
        iA^{*} f_{\pm}(x,k) + \nu(k)  \widetilde{f}_{\pm}(x,k) \, .
    \end{pmatrix}
\end{align*}
Applying the intertwining relations \eqref{eq:intertwining} to $g_{\pm}$, we have
\begin{align*}
    g_{\pm}(x,k) = \frac{1}{2}S\begin{pmatrix}
        (\nu(k) \pm k \mp i\minf)f_{\pm}(x,k)\\
        (\nu(k) \pm k \pm i\minf)\widetilde{f}_{\pm}(x,k)
    \end{pmatrix}\, .
\end{align*}
This form makes the asymptotics clear as
\begin{align*}
    \lim_{x\to+\infty}g_{+}(x,k) e^{-ikx} &=\lim_{x\to+\infty}\frac{1}{2}S\begin{pmatrix}
        (\nu(k) +k-i\minf)f_{+}(x,k)\\
        (\nu(k) +k+i\minf)\widetilde{f}_{+}(x,k)
    \end{pmatrix}e^{-ikx}\\
    &=\frac{1}{2}S\begin{pmatrix}
        (\nu(k) +k-i\minf)\\
        (\nu(k) +k+i\minf)
    \end{pmatrix}  =\frac{1}{\sqrt{2}}\begin{pmatrix}
        \nu(k) +k\\
        \minf
    \end{pmatrix} .
\end{align*}
The asymptotics for $g_{-}(x,k)$ can be similarly checked.
\end{proof}

\subsection{Computation of the Dirac Wronskian}\label{sec:1d_dirac_W}
Recall the definition of the Schr\"odinger Wronskian  $W_{H}(k)$: 
\begin{align*}
    &W_{H}(k) = W[f_{+}(\cdot,k),f_{-}(\cdot,k)] \, . 
\end{align*}
Now we define the corresponding Dirac Wronskian  $W_{\Dx}(k)$ associated to the Jost solutions of the Dirac operator $\Dx$:
\begin{align*}
      &W_{\Dx}(k) \equiv  \det[g_{+}(\cdot,k),g_{-}(\cdot,k)].
\end{align*}
\begin{theorem}\label{thm:1d_wronskians}
We have the following relation between the Schr\"odinger and Dirac Wronskians:
\begin{align*}
    W_{\Dx}(k) =\frac{\minf\nu(k) }{2(ik-\minf)}W_{H}(k) \, .
\end{align*}
\end{theorem}

\begin{proof}
The proof is a direct calculation using the intertwining relations \eqref{eq:intertwining} and the definition \eqref{def:AA*} of $A^*$.
\begin{align*}
    W_{\Dx}(k)&=\frac{\det(S)}{4}\left[(\nu(k) +k-i\minf)(\nu(k) -k-i\minf)f_{+}(\cdot,k)\widetilde{f}_{-}(\cdot,k) \right.\\
    &\left.- (\nu(k) -k+i\minf)(\nu(k) +k+i\minf)f_{-}(\cdot,k)\widetilde{f}_{+}(\cdot,k) \right ]\\ 
    &= -\frac{\minf\nu(k) }{2}\left( f_{+}(\cdot,k)\widetilde{f}_{-}(\cdot,k)+f_{-}(\cdot,k)\widetilde{f}_{+}(\cdot,k)  \right)\\
    &= \frac{\minf\nu(k) }{2}\left(f_{+}(\cdot,k)\frac{A^{*}f_{-}(\cdot,k)}{-ik+\minf}- f_{-}(\cdot,k)\frac{A^{*}f_{+}(\cdot,k)}{-ik+\minf} \right)\\
    &= \frac{\minf\nu(k) }{2(ik-\minf)}\left(f_{+}(\cdot,k)f_{-}'(\cdot,k) -f_{-}(\cdot,k)f_{+}'(\cdot,k) \right)= \frac{\minf\nu(k) }{2(ik-\minf)}W_{H}(k)\, .
\end{align*}
\end{proof}

\subsection{The Resolvent and Time Evolution Operator}\label{sec:1d_resolvent}
For $z\notin \sigma(\Dx)$, define the resolvent $R_{\Dx}(z)=(\Dx-z)^{-1}$. Let $z=\nu(k)$ with $\im(k)>0$. The definition of the Jost functions $g_{\pm}(x,k)$, by \eqref{eq:DJost_evalue}, can be extended to $k\in\C$ with $\im(k)>0$ in a natural way. Then $g_{+}(x,k)$ decays exponentially as $x\to+\infty$ and $g_{-}(x,k)$ decays exponentially as $x\to-\infty$, by the asymptotics \eqref{eq:DJost_asympt}. The function $R_{\Dx}(\nu(k))(x,y)$ satisfies the equation
\begin{align}\label{eq:Rdx_kernelDef}
    (\Dx-\nu(k))R_{\Dx} = \delta(x-y)\, .
\end{align}
To determine the kernel, we use the ODE method of undetermined coefficients: by the decay conditions at $\pm \infty$,  the resolvent kernel must have the following form
\begin{align*}
    R_{\Dx}(\nu(k))(x,y) = \begin{cases}
        g_{+}(x,y)a(y,k)^{\top} & x > y \, ,\\
        g_{-}(x,y)b(y,k)^{\top} & x < y \, ,
    \end{cases}
\end{align*}
for some undetermined functions $a,b$. Integrating the equation above satisfied by $R_{\Dx}(z)(x,y)$, \eqref{eq:Rdx_kernelDef}, along the line $x=y$ yields that these coefficients are in fact
\begin{align*}
    R_{\Dx}(\nu(k))(x,y) = \frac{i}{W_{\Dx}(k)}\begin{cases}
        g_{+}(x,k)g_{-}(y,k)^{\top}\sigma_1 & x > y \, , \\
        g_{-}(x,k)g_{+}(y,k)^{\top}\sigma_1 & x < y \, ,
    \end{cases}
\end{align*}
where $W_{\Dx}(k)=\det[g_{+}(\cdot,k),g_{-}(\cdot,k)]$. Note that this determinant is independent of $x$ by Abel's theorem (concerning the evolution of the Wronskian in $x$ for systems of ODEs \cite{coddington1955theory}). The limiting absorption principle as $\im(k)\to 0$ follows from the identity
\begin{align*}
    R_{\Dx}(z) = (\Dx+z)S\begin{pmatrix}
        (H+\minf^2 - z^2)^{-1} & 0\\
        0 & (\widetilde{H}+\minf^2 - z^2)^{-1}
    \end{pmatrix}S^* \, ,
\end{align*}
and the corresponding limiting absorption principle for Schr\"odinger operators. Let $R_{\pm}(k)(x,y) = R_{\Dx}(\nu(k\pm i0))$. These limiting Dirac resolvents satisfy, for all $k\in \R\setminus\{0\}$,
\begin{align}\label{eq:ResKerR}
    R_{\pm}(k)(x,y) = \frac{i}{W_{\Dx}(\pm k)}\begin{cases}
        g_{+}(x,\pm k)g_{-}(y,\pm k)^{\top}\sigma_1 & x > y \, ,\\
        g_{-}(x,\pm k)g_{+}(y,\pm k)^{\top}\sigma_1 & x < y \, .
    \end{cases} 
\end{align}
Let $P_{\rm c}^+(\Dx)$ be the projection onto the positive part of the continuous spectrum of $\Dx$. Then by Stone's formula for $\gamma > 0$,
\begin{align*}
    \left[e^{-it\Dx}\langle \Dx\rangle^{-\gamma}P_c^+(\Dx)\right](x,y) = \frac{1}{2\pi i}\int_{\minf}^{\infty} e^{-itz}\langle z\rangle ^{-\gamma}\left[R_{\Dx}(z +i0)-R_{\Dx}(z -i0)\right](x,y)dz.
\end{align*}
Making the change of variables $z = \nu(k)$ we arrive at the formula
\begin{align*}
    \left[e^{-it\Dx}\langle \Dx\rangle^{-\gamma}P_c^+(\Dx)\right](x,y) = \frac{1}{2\pi i}\int_{0}^{\infty} e^{-it\nu(k)}\left[R_{+}(k)-R_{-}(k)\right](x,y)\frac{k dk}{\nu(k)^{\gamma+1}},
\end{align*}
where we have harmlessly replaced $\langle \nu(k)\rangle$ by $\nu(k)$. Using the formulas for $R_{\pm}(k)$ in \eqref{eq:ResKerR}, and making the change $k\to-k$ in the second integral, we observe that the propagator kernel can be expressed in terms of the Jost functions $g_{\pm}(x,k)$ by
\begin{align*}
    \left[e^{-it\Dx}\langle \Dx\rangle^{-\gamma}P_c^+(\Dx)\right](x,y) = \frac{1}{2\pi}\int_{\R}e^{-it\nu(k)}\frac{g_{+}(x,k)g_{-}(y,k)^{\top}\sigma_1}{W_{\Dx}(k)}\frac{k dk}{\nu(k)^{\gamma+1}},\quad x > y ,\\
    \left[e^{-it\Dx}\langle \Dx\rangle^{-\gamma}P_c^+(\Dx)\right](x,y) = \frac{1}{2\pi}\int_{\R}e^{-it\nu(k)}\frac{g_{-}(x,k)g_{+}(y,k)^{\top}\sigma_1}{W_{\Dx}(k)}\frac{k dk}{\nu(k)^{\gamma+1}},\quad x < y .
\end{align*}

\subsection{Proofs of Main Theorems}\label{sec:1d_proofs}
In this section we state and prove our main results concerning the one-dimensional Dirac operator $\Dx$, see \eqref{eq:1d_dirac}.

\begin{theorem}\label{thm:1d_unweighted}
Assume that $m^2(x)-\minf^2,  m'(x) \in L^1_1$ when  $W_\Dx(0)\neq 0$, and   $m^2(x)-\minf^2,  m'(x) \in L^1_2$ when $W_{\Dx}(0)=0$. Let $P_c(\Dx)$ be the projection onto the absolutely continuous part of the spectrum of $\Dx$. For any $\gamma> 3/2$, 
\begin{align*}
    \|e^{-it\Dx}\langle \Dx\rangle^{-\gamma}P_c(\Dx)\|_{L^1\to L^{\infty}}\lesssim \langle t\rangle ^{-1/2}.
\end{align*}

\end{theorem}
\begin{proof}
Fix $\eps>0$ and assume WLOG that $x<y$. We claim that under the hypotheses
$$F(x,y,k)=\dfrac{k g_-(x,k)g_+(y,k)^{\top}\sigma_1}{W_{\Dx}(k)\nu(k)^{1+\eps}},$$ 
is in the Wiener algebra $A(\R)$, as a function of $k$. Moreover, $\displaystyle\sup_{x<y}\| F(x,y,\cdot )\|_{A(\R)}\lesssim 1$.
Recall from Theorem~\ref{thm:1d_wronskians} that
\begin{align*}
    W_{\Dx}(k) = \frac{\minf \nu(k)}{2(ik-\minf)}W_{H}(k),
\end{align*}
where $W_{H}(k)$ is the scalar Wronskian associated to the Schr{\"o}dinger Hamiltonian $H$ given in \eqref{eq:1d_Schrodinger}. Thus
\begin{align*}
    F(x,y,k) = \dfrac{2(ik-\minf)k g_-(x,k)g_+(y,k)^{\top}\sigma_1}{W_{H}(k)\nu(k)^{2+\eps}}
\end{align*}
\subsubsection*{Case 1: $x<0<y$}
Here $F(x,y,k)$ can, up to constants, be written as a product of Schr\"odinger Jost functions $\varphi_{-}(x,k)\varphi_{+}(y,k)$ (each in $A_1(\R)$ uniformly in $x<0<y$, see \eqref{eq:DT_bounds}), times functions of the form
\begin{align*}
    e^{ik(y-x)},    \quad \frac{k}{W_{H}(k)}, \quad h(k) = \frac{(ik-\minf)(\nu(k)-k)^{n}(\nu(k)+k)^{l}}{\nu(k)^{2+\eps}},\qquad n,l = 0,1.
\end{align*}
The term $\frac{k}{W_H(k)}\in A_1(\R)$ in either the resonant or non-resonant case by Lemma~\ref{lem:wronratio} by writing
\begin{align*}
    \frac{k}{W_H(k)} = \frac{k}{\lambda(k)}\frac{\lambda(k)}{W_H(k)}.
\end{align*}
Each function $h(k)\in A(\R)$ for each choice of $n,l$, as they can be broken into a linear combination of functions  \begin{align*}
    \nu(k)^{-2-\eps},\quad k\nu(k)^{-2-\eps}, \quad\nu(k)^{-1-\eps},\quad k\nu(k)^{-1-\eps},\quad\nu(k)^{-\eps},
\end{align*} whose Fourier transform is a Bessel potential or its derivative; see~\cite[Chapter 5]{Stein_SingularIntegrals} (this also follows from Lemma~\ref{lem:hhat} below). Hence, since by the Fourier multiplication rule, a product of functions in $A(\R)$ and $A_1 (\R)$ is in $A(\R)$, we have that $F(x,y,k)\in A(\R)$. 
\subsubsection*{Case 2: $0<x<y$}
Recall that $\psi_-(x,k)\equiv\frac{k}{W_H(k)}\varphi_-(x,k)$. In this case $F(x,y,k)$ can, up to constants, be written as a product of $\psi_{-}(x,k)\varphi_{+}(y,k)$ times functions of the form
\begin{align*}
    e^{ik(y-x)}, \quad h(k) = \frac{(ik-\minf)(\nu(k)-k)^{n}(\nu(k)+k)^{l}}{\nu(k)^{2+\eps}},\qquad n,l = 0,1.
\end{align*}
By Lemma~\ref{lem:badside}, $\psi_{-}(x,k)\varphi_{+}(y,k)$ is in $\mF TV(\R)$ uniformly in $0 < x,y$. Once again, $F(x,y,k)\in A(\R)$.  
The case that $x<y<0$ is treated similarly, except that we use $\psi_{+}(y,k)$. 

We have proven that $\displaystyle\sup_{x<y}\| F(x,y,\cdot )\|_{A(\R)}\lesssim 1$, and we may now proceed to the main dispersive estimate. Having assumed $\gamma > 3/2$ choose $\eps>0$ such that $\gamma -\eps>3/2$. For $x<y$ we have
\begin{align*}
    [e^{-it\Dx}\langle \Dx\rangle^{-\gamma}P_c^+(\Dx)](x,y) &= \frac{1}{2\pi}\int_{\R}e^{-it\nu(k) }\frac{g_-(x,k)g_+(y,k)^{\top}\sigma_1}{W_{\Dx}(k)}\frac{k}{\nu(k)^{\gamma+1}}dk \, .
\end{align*}
Thus by the Fourier multiplication formula and using Lemma~\ref{lem:KGdisp}, we have  
\begin{align*}
\begin{split}
    \|[e^{-it\Dx}\langle \Dx\rangle^{-\gamma}P_c^+(\Dx)](x,y)\|_{L^{\infty}} &\leq\Big\| \mathcal F^{-1} \Big(\frac{e^{-it\nu(k) } }{\nu(k) ^{\gamma-\epsilon} }\Big) \Big\|_{L^\infty} \Big\| F(x,y,\cdot)\Big\|_{A(\R)}\lesssim \langle t\rangle^{-1/2} \, ,
\end{split}
\end{align*}
where the factor of $\la t\ra^{-1/2}$ comes from an application of Lemma~\ref{lem:KGdisp}. The cases of $y<x$ for the positive spectrum as well as the negative spectral bound come from analogous expressions for the kernel of the propagator.
\end{proof}

In the nonresonant case we can improve the rate of decay at the expense of additional spatial weights.
\begin{theorem}\label{thm:1dWeighted}
Assume that $m^2(x)-\minf^2,  m'(x) \in L^1_3$ and  $W_{\Dx}(0)\neq 0$. Let $P_c(\Dx)$ be the projection onto the absolutely continuous part of the spectrum of $\Dx$. For any $\gamma> 3/2$,
\begin{align*}
    \|e^{-it\Dx}\langle \Dx\rangle^{-\gamma}P_c(\Dx)\|_{L_1^1\to L_{-1}^{\infty}}\lesssim \langle t\rangle ^{-3/2}.
\end{align*}
\end{theorem} 

\begin{proof}
We once again only prove the case that $x<y$ and positive spectrum in detail. Assume that $W_{\Dx}(0)\neq 0$ and recall that for $x<y$
$$
    [e^{-it\Dx}\langle \Dx\rangle^{-\gamma}P^+_c(\Dx)](x,y)  = \frac{1}{2\pi}\int_{\R}e^{-it\nu(k) }\frac{g_-(x,k)g_+(y,k)^{\top}\sigma_1}{W_{\Dx}(k)}\frac{k}{\nu(k)^{\gamma+1}}dk.
$$
Writing $e^{-it\nu(k)}\frac{k}{\nu(k)}=\frac{i}{t}\partial_k e^{-it\nu(k)}$ and integrating by parts, we obtain 
\begin{multline*}
    [e^{-it\Dx}\langle \Dx\rangle^{-\gamma}P^+_c(\Dx)](x,y)  = \frac{1}{2\pi i t }\int_{\R}e^{-it\nu(k) } \partial_k\Big(\frac{g_-(x,k)g_+(y,k)^{\top}\sigma_1}{W_{\Dx}(k)}\frac{1}{\nu(k)^{\gamma }}\Big)dk\\
    =  \frac{1}{2\pi i t } \int_{\R} \mathcal F^{-1} \big(e^{-it\nu(k)}\nu(k)^{-\frac32-}\big)(\xi) \mathcal F\Big(\nu(k)^{\frac32+}\partial_k\Big(\frac{g_-(x,k)g_+(y,k)^{\top}\sigma_1}{W_{\Dx}(k)}\frac{1}{\nu(k)^{\gamma }}\Big)\Big)(\xi)  d\xi.  
\end{multline*}
As in the proof of the unweighted estimate (Theorem \ref{thm:1d_unweighted}), using the Fourier multiplication formula and applying the one-dimensional Klein Gordon estimate (Lemma \ref{lem:KGdisp}), the weighted estimate is reduced to proving the following  inequality
\begin{equation}\label{eq:pkhk}
\Big\|  \nu(k)^{\frac32+}\partial_k\Big(\frac{g_-(x,k)g_+(y,k)^{\top}\sigma_1}{W_D(k)}\frac{1}{\nu(k)^{\gamma }}\Big) \Big\|_{A(\R)}\lesssim \langle x\rangle \langle y\rangle,
\end{equation}
To prove \eqref{eq:pkhk}, recall that
$$
    W_{\Dx}(k) = \frac{  \nu(k)}{2(ik-\minf)}W_{H}(k) .
$$  
Therefore the entries of $$
\nu(k)^{\frac32+}\partial_k\Big(\frac{g_-(x,k)g_+(y,k)^{\top}\sigma_1}{W_{\Dx}(k) \nu(k)^{\gamma }} \Big)$$
are linear combinations of
$$ 
\nu(k)^{\frac32+}\partial_k\Big(\frac{ e^{-ik(x-y)}    \varphi_-(x,k)\varphi_+(y,k)}{W_{H}(k)} \ \frac{ (ik-\minf) (\nu(k)-k)^{n}(\nu(k)+k)^{l} }{ \nu(k)^{\gamma+1}} \Big),\quad n,l=0,1. 
$$
With this, the inequality \eqref{eq:pkhk} follows from the Leibniz rule and the following (to be justified) bounds
\begin{align}
\label{Wcl1}
&\big\| \partial_k^{n}[e^{ik(y-x)}    \varphi_-(x,k)\varphi_+(y,k) ] \big\|_{\mF TV(\R)}\les  \la x\ra \la y\ra, \quad n=0,1\\ 
\label{Wcl2}
&\Big\|   \frac{   ik-\minf  }{W_{H}(k)}   \Big\|_{A_1(\R)}   \les 1,\quad \Big\| \partial_k  \frac{   ik-\minf  }{W_{H}(k)}   \Big\|_{A_1(\R)}   \les 1, \\ \label{Wcl3}
&\Big\|  \frac{   (\nu(k)-k)^{n}(\nu(k)+k)^{l} }{ \nu(k)^{\gamma-\frac12-}}  \Big\|_{A(\R)}\les 1, \quad n,l=0,1,\\ \label{Wcl4}
&  \Big\|\nu(k)^{\frac32+} \partial_k \Big(\frac{   (\nu(k)-k)^{n}(\nu(k)+k)^{l} }{ \nu(k)^{\gamma+1}}  \Big)\Big\|_{A(\R)}\les 1, \quad n,l=0,1.
\end{align} 
Note that the bounds  \eqref{Wcl3}, \eqref{Wcl4} concern explicit functions, and therefore follow from Lemma~\ref{lem:hhat} below, provided that $\gamma>3/2$.

The first bound in $\eqref{Wcl2}$ follows from Lemma~\ref{lem:wronratio} by writing
$$
\frac{ik-\minf}{W_{H}(k)}= \frac{ik-\minf}{\lambda(k)} \frac{\lambda(k) }{W_{H}(k)}\in A_1(\R).$$
The second bound in \eqref{Wcl2} also follows from
Lemma~\ref{lem:wronratio} by writing 
$$
\frac{(ik-\minf)\partial_kW_{H}(k) }{W_{H}(k)^2}= \frac{ik-\minf}{\lambda(k)} \frac{\lambda(k)^2 }{W_{H}(k)^2} \frac{\partial_kW_{H}(k) }{\lambda(k)}\in A_1(\R) 
$$ 
and
$$
\frac{1}{W_H(k)}=\frac{\lambda(k)}{W_H(k)}\frac{1}{\lambda(k)}\in A(\R).$$
It remains to prove \eqref{Wcl1}.
Both inequalities follow from \eqref{eq:DT_bounds} when $x<0<y$.

When $0<x<y$, by Lemma~\ref{lem:badside} and \eqref{eq:DT_bounds}, and noting that the oscillatory term does not affect the norm, we have
$$
 \Big\|  e^{ik(y-x)} \varphi_-(x,k)\varphi_+(y,k)  \Big\|_{\mF TV(\R)}\les  \la x\ra.
$$
Similarly, for the $k$-derivative, we have the bound $|x-y|\la x\ra \les \la x\ra \la y\ra$ when the derivative hits the oscillatory term, and   the bound $\la x\ra^2  \leq \la x\ra \la y\ra$     when  the derivative hits the modified Jost functions.
\end{proof}

The following is a well-known bound that suffices for our purposes. We include a proof for completeness.
\begin{lemma}\label{lem:hhat}
Let $f$ be a $C^\infty$ function on $\R^n$ satisfying for some $\eps>0$
$$|\nabla^N f(k)|\leq C_{\eps,N} \la k\ra^{-\eps-N} \, ,\qquad \forall N=0,1,2,\ldots \, .$$
Then $f$ and all of its derivatives are in $A(\R^n)$.
\end{lemma}
\begin{proof} The claim for the derivatives follows from the claim for $f$ since they satisfy the same bounds. Also note that for any compactly supported smooth $\chi_0$, $f\chi_0\in A(\R^n)$ as it is Schwartz.  

Let $\chi $ be a smooth cutoff for the set $\{k\in\R^n:|k|\approx 1\}$.   By repeated integration  by parts, for each $j,N=0,1,2,\ldots,$ we have $$ |\mF\big(f(\cdot)\chi(\cdot 2^{-j})\big)(\xi)|\les |\xi|^{-N}  2^{-j(\eps+N)} 2^{jn}.$$  
Using this bound for $N=n+1 $ and summing over $j$, we see that $\widehat f\in L^1 ({|\xi|\gtrsim 1})$.

For $|\xi|\les 1$, using the bound above for $N=0$ and $N=n$, we have
$$ \sum_{j=0}^\infty |\mF\big(f(\cdot)\chi(\cdot 2^{-j})\big)(\xi)| \les \sum_{j=0}^\infty 2^{-j\eps}\min(2^{jn},|\xi|^{-n})\les |\xi|^{\eps-n},$$
which is in $L^1({|\xi|\les 1})$.
\end{proof}

\section{Dispersive Estimates for two-dimensional Dirac operators with bounded domain wall}\label{sec:2d_bounded_estimates}
In this section we return to the study of the two-dimensional Dirac operator
\begin{align*}
    D = -i\partial_{x_1}\sigma_1 -i\partial_{x_2}\sigma_2 + m(x_2)\sigma_3 \, ,
\end{align*}
with a bounded domain wall. In Sec.\ \ref{sec:2dbd_spec} we use the fact that $D$ is translation invariant in the variable $x_1$ to study a fibered family of one-dimensional Dirac operators, $\hatD$, defined in \eqref{eq:fibered_dirac}. From here on, the proof progresses in much the same way as it did in the one-dimensional settings of Sec.\ \ref{sec:1d_estimates}.  In Secs.\ \ref{sec:2d_dirac_jost} and \ref{sec:2d_dirac_wron} we develop the scattering theory associated with this fibered family, once again taking advantage of the Schr\"odinger scattering theory developed in Sec.\ \ref{sec:1d_shrodinger}. In Sec.\ \ref{sec:2d_resolvent} we obtain formulas for the spectral measure associated to the operator $D$. This is more delicate than the analogous steps for the one-dimensional equation: Although $D$ has purely absolutely continuous spectrum, each fiber $\widehat{D}(k_1)$ has nontrivial point spectrum. Finally, the proofs of Theorems \ref{thm:2d_main_theorem} and \ref{thm:2dweighted} are given in Sec.\ \ref{sec:2d_proofs}.

\subsection{Spectrum}\label{sec:2dbd_spec}
We require precisely the same hypotheses on the function $m(x)$ as in Sec.\ \ref{sec:1d_estimates}; specifically that $m^2(x)-\minf^2,m'(x)\in L^1_1(\R)$. In order to derive dispersive estimates associated to this operator we need to develop an analogous Jost function framework. As the operator is translation invariant in the variable $x_1$, we take the Fourier transform in this variable and study the family of fibered Dirac Hamiltonians $\hatD$.

For each $k_1\in\R$, the $x_1\to k_1$ Fourier transform  of the operator $D$ is 
\begin{align}\label{eq:fibered_dirac}
    \hatD = k_1\sigma_1 -i\partial_{x_2}\sigma_2+m(x_2)\sigma_3 \, .
\end{align}
Let $S = \dfrac{1}{\sqrt{2}}\begin{pmatrix}
    1 & 1\\
    1 & - 1
\end{pmatrix}$. Since  $S^{-1}=S$,  
\begin{align}\label{eq:SDS2}
    S\hatD S = \begin{pmatrix}
        k_1 & A\\ A^{*} & -k_1
    \end{pmatrix},\quad   S\hatD^2S = \begin{pmatrix}
        H+\minf^2+k_1^2& 0\\ 0 & \widetilde{H}+\minf^2+k_1^2
    \end{pmatrix} \, ,
\end{align}
where $A, A^*, H,$ and $\widetilde{H}$ are all as in Section~\ref{sec:SUSY}. By Lemma \ref{lem:SUSY_prop} we have
\begin{align*}
    \sigma(H) = \{\lambda_j\}_{j=1}^N\cup[0,\infty) \, ,
\end{align*}
i.e., the spectrum of $H$ is composed of absolutely continuous spectrum on the non-negative real line and a finite number of simple negative eigenvalues. Moreover, we have $\la \psi,H\psi\ra > -\minf^2\|\psi\|^2$ and hence each eigenvalue $\lambda_j > -\minf^2$. For notational purposes, let $\{\phi_j\}_{j=0}^N$ and $\{\pi_j\}_{j=1}^N$ be the eigenfunctions of $\widetilde{H}$ and $H$, respectively. In particular
\begin{align*}
    \widetilde{H}\phi_0 = -\minf^2\phi_0 \, ,\qquad \widetilde{H}\phi_j = \lambda_j\phi_j \, ,\qquad H\pi_j=\lambda_j\pi_j \, ,\qquad j=1,\dots, N.
\end{align*}
\begin{lemma}\label{lem:hatD_spectrum}
For each $k_1\in\R$ the spectrum of $\hatD$ is composed of
\begin{itemize}
    \item Absolutely continuous spectrum in $(-\infty,-\sqrt{\minf^2+k_1^2}]\cup[+\sqrt{\minf^2+k_1^2},\infty)$.
    \item Finitely many simple eigenvalues in the gap $\left\{\nu_j(k_1)\right\}_{j=-N}^N$, where
    \begin{align*}
    \nu_j(k_1)=\begin{cases}
        -\sqrt{\minf^2+k_1^2+\lambda_{|j|}} & -N \leq j < 0\\
        -k_1 & j = 0\\
         +\sqrt{\minf^2+k_1^2+\lambda_j} & 0 < j \leq N
    \end{cases}.
    \end{align*}
    \end{itemize}
    Moreover, the normalized eigenfunctions are given by $\Phi_{j}(k_1,x_2)$ for $j=-N,\dots,N$.
    \begin{align*}
        &\Phi_0(k_1,x_2) = \frac{1}{\sqrt{2}}\begin{pmatrix}
            +\phi_0(x_2)\\ -\phi_0(x_2)
        \end{pmatrix},\quad \hatD \Phi_0(k_1,x_2) = -k_1 \Phi_0(k_1,x_2).\\
        &\Phi_{j}(k_1,x_2) = \frac{1}{\sqrt{4\nu_j(k_1)(\nu_j(k_1)+k_1})}\begin{pmatrix}
            (k_1+\nu_j(k_1))\phi_{|j|}(x_2) +\sqrt{m_{\infty}^2 + \lambda_{|j|}}\pi_{|j|}(x_2)\\
            (k_1+\nu_j(k_1))\phi_{|j|}(x_2) -\sqrt{m_{\infty}^2+\lambda_{|j|}}\pi_{|j|}(x_2)
        \end{pmatrix}.\\
        &\hatD\Phi_{j}(k_1,x_2) = \nu_{j}(k_1)\Phi_{j}(k_1,x_2).
    \end{align*}
\end{lemma}

\begin{proof}
The statement about the absolutely continuous spectrum follows from the identity
\begin{align*}
    (\hatD-z)^{-1} = (\hatD+z)S\begin{pmatrix}
        (H+\minf^2+k_1^2-z^2)^{-1} & 0\\
        0 & (\widetilde{H}+\minf^2 + k_1^2-z^2)^{-1}
    \end{pmatrix}S.
\end{align*}
As for the point spectrum, $S\hatD^2 S$ has eigenvalues $k_1^2$ and $\minf^2+k_1^2+\lambda_j$ with eigenspaces
\begin{align*}
        E_{0} &= \ker(\hatD^2-k_1^2) = span\left\{\begin{pmatrix}
            \phi_0(x_2) \\ -\phi_0(x_2)
        \end{pmatrix}\right\} ,\\
        E_{j} &=\ker(\hatD^2-(\minf^2+k_1^2+\lambda_j))= span\left\{\begin{pmatrix}
            \phi_{j}(x_2) \\ -\phi_{j}(x_2)
        \end{pmatrix}, \begin{pmatrix}
            \pi_{j}(x_2) \\ \pi_{j}(x_2)
        \end{pmatrix}\right\} ,\quad j=1,\dots,N.
\end{align*}
As $\hatD$ and $\hatD^2$ commute, $\hatD$ preserves the eigenspaces $E_j$. A direct computation shows that the eigenvalues and eigenfunctions of $\hatD$ are as stated in the lemma. 
\end{proof}
Using the lemma above, we define the projections  onto the bulk and edge parts of the spectrum of $D$, denoted by $P_{\rm edge}$ and $P_{\rm bulk}$, and illustrated in Fig.\ \ref{fig:spectrum}. The image of each projection evolves in a qualitatively different way, as shown by the dispersive decay estimates, Theorems \ref{thm:2d_main_theorem} and \ref{thm:2dweighted}.
\begin{definition}\label{def:projections}
For $j=-N,\dots,N$, define projections $P_j$ on $L^2(\R^ 2;\C^2)$ by 
\begin{equation}\label{eq:Pj2dDef}
(P_j f) (x_1, x_2) = \int\limits_{\R} e^{ik_1 x_1} \langle \Phi_{j}(k_1, \cdot) , \widehat{f}(k_1,\cdot)\rangle_{L^2 (\R_{x_2};\C^2)} \Phi_{j}(k_1, x_2) \, \frac{dk_1}{2\pi} \, , \qquad f\in L^2 (\R ^2; \C^2) \, .
\end{equation}
We also define ``bulk" and ``edge" projections $P_{\rm bulk}$ and $P_{\rm edge}$.
\begin{equation}\label{eq:Pbulk2d}
P_{\rm bulk} = Id - P_{\rm edge} \,  ,\qquad P_{\rm edge} \equiv  \sum_{ j=-N}^{N}  P_j \, .
\end{equation}
\end{definition}
\subsection{Two-dimensional Dirac Jost functions}\label{sec:2d_dirac_jost}
We can now construct the Jost functions which are the two-dimensional analogs of \eqref{eq:1d_gpm_def}.
\begin{theorem}
Let $k_1,k_2\in\R$, $k_2\neq 0$, and $S = \dfrac{1}{\sqrt{2}}\begin{pmatrix}
    1 & 1\\
    1 & - 1
\end{pmatrix}$ and 
\begin{align}\label{eq:2d_nu}
    \nu(k_1,k_2) = \sqrt{m_{\infty}^2+k_1^2+k_2^2}.
\end{align}
Define $g_{\pm}(x_2,k_1,k_2)$ through the formulas
\begin{align*}
    g_{\pm}(x_2,k_1,k_2) &= \frac{1}{2}(\hatD+\nu(k_1,k_2))S\begin{pmatrix}
        f_{\pm}(x_2,k_2)\\ \widetilde{f}_{\pm}(x_2,k_2)
    \end{pmatrix}.
\end{align*}
Then $g_{\pm}(x_2,k_1,k_2)$ are Jost functions for $\hatD$ in the sense that they satisfy the equation
\begin{align*}
    (\hatD-\nu(k_1,k_2))g_{\pm} = 0,
\end{align*}
along with the one-sided asymptotics
\begin{align*}
    g_{\pm }(x_2,k_1,k_2) &\sim \frac{1}{\sqrt{2}}\begin{pmatrix}
        \pm \minf + \nu(k_1,k_2) \\
        k_1\pm ik_2
    \end{pmatrix}e^{\pm ik_2x_2},\quad x_2\to \pm \infty.
\end{align*}
\end{theorem}
\begin{rmk}
The Jost functions $g_{\pm}(x_2,k_1,k_2)$ are associated with the positive part of the spectrum of $\hatD$. As in the one-dimensional case, Jost functions associated to the negative part of the spectrum can be constructed analogously.
\end{rmk}
\begin{proof}
Using \eqref{eq:SDS2}, we can compute
\begin{align*}
    S(\hatD-\nu(k_1,k_2))g_{\pm}&=\frac{1}{2}S(\hatD^2-\nu(k_1,k_2)^2)S\begin{pmatrix}
        f_{\pm}(x_2,k_2)\\ \widetilde{f}_{\pm}(x_2,k_2)\end{pmatrix}\\
        &= \frac{1}{2}\begin{pmatrix}
            H - k_2^2 & 0 \\
            0 & \widetilde{H} -k_2^2
        \end{pmatrix}\begin{pmatrix}
        f_{\pm}(x_2,k_2)\\ \widetilde{f}_{\pm}(x_2,k_2)\end{pmatrix}=0 \,,
\end{align*}
with $f_{\pm}(x,k)$ and $\widetilde{f}_{\pm}(x,k)$ as in \eqref{eq:HJost}. As $S$ is invertible we have $(\hatD-\nu(k_1,k_2))g_{\pm}=0$. Thus, using \eqref{eq:SDS2} again, we can rewrite $g_{\pm}(x_2,k_1,k_2)$ as
\begin{align*}
    g_{\pm}(x_2,k_1,k_2) &= \frac{1}{2}(\hatD+\nu(k_1,k_2))S\begin{pmatrix}
          f_{\pm}(x_2,k_2)\\ \widetilde{f}_{\pm}(x_2,k_2)
    \end{pmatrix}\\
    &= \frac{1}{2}S\begin{pmatrix}
        k_1 + \nu(k_1,k_2) & A\\ A^{*} & -k_1+\nu(k_1,k_2)
    \end{pmatrix}\begin{pmatrix}
          f_{\pm}(x_2,k_2)\\ \widetilde{f}_{\pm}(x_2,k_2)
    \end{pmatrix}\\
    &=\frac{1}{2} S\begin{pmatrix}
        (k_1+\nu(k_1,k_2))f_{\pm}(x_2,k_2) + A\widetilde{f}_{\pm}(x_2,k_2)\\
        A^{*}f_{\pm}(x_2,k_2) + (-k_1+\nu(k_1,k_2))\widetilde{f}_{\pm}(x_2,k_2)
    \end{pmatrix}
\end{align*}
By the intertwining relations \eqref{eq:intertwining} 
\begin{align}\label{eq:gpmdef}
    g_{+}(x_2,k_1,k_2) &= \frac{1}{2}S\begin{pmatrix}
        (\nu(k_1,k_2)+k_1+ik_2+\minf)f_{+}(x_2,k_2)\\
        (\nu(k_1,k_2)-k_1-ik_2+\minf)\widetilde{f}_{+}(x_2,k_2)
    \end{pmatrix}\\
    g_{-}(x_2,k_1,k_2) &=\frac{1}{2}S\begin{pmatrix}
        (\nu(k_1,k_2)+k_1-ik_2-\minf )f_{-}(x_2,k_2)\\
        (\nu(k_1,k_2)-k_1 +ik_2-\minf)\widetilde{f}_{-}(x_2,k_2)
    \end{pmatrix}
\end{align}
The asymptotics are checked directly using those satisfied by $f_{\pm}$ and $\widetilde{f}_{\pm}.$
\end{proof}
\subsection{Wronskian Calculation}\label{sec:2d_dirac_wron}
Once again we can relate the Wronskian of $\hatD$ to the Wronskian of the one-dimensional Schr\"odinger operator \eqref{eq:1d_Schrodinger}.
\begin{theorem}\label{thm:2d_wronskians}
 We define the Dirac Wronskian, $W_{\widehat{D}}(k_1,k_2)$,   associated to the Jost solutions of the Dirac operator $\hatD$:
\begin{align*}
      W_{\widehat{D}}(k_1,k_2) = \det[g_{+}(\cdot,k_1,k_2),g_{-}(\cdot,k_1,k_2)].
\end{align*}
We have the following relation  between the Schr\"odinger and Dirac Wronskians:
\begin{align}\label{eq:2d_wroskian}
    W_{\widehat{D}}(k_1,k_2) &=\frac{i\nu(k_1,k_2)k_2-k_1\minf}{2(ik_2-\minf)}W_{H}(k_2).
\end{align}
\end{theorem}
\begin{proof}
The proof is a direct calculation similar to that of Theorem \ref{thm:1d_wronskians}.
\begin{align*}
    W_{\widehat{D}}(k_1,k_2) &= \frac{\det(S)}{4}\left[(\nu(k_1,k_2)+k_1+ik_2+\minf)(\nu(k_1,k_2)-k_1+ik_2-\minf)f_{+}(x_2,k_2)\widetilde{f}_{-}(x_2,k_2)\right.\\
    &\left.-(\nu(k_1,k_2)+k_1-ik_2-\minf)(\nu(k_1,k_2)-k_1-ik_2+\minf)f_{-}(x_2,k_2)\widetilde{f}_{+}(x_2,k_2) \right]\\
    &=\frac{(k_1\minf-i\nu(k_1,k_2)k_2)}{2}\left[f_{+}(x_2,k_2)\widetilde{f}_{-}(x_2,k_2)+f_{-}(x_2,k_2)\widetilde{f}_{+}(x_2,k_2) \right]\\
    &=\frac{(k_1\minf-i\nu(k_1,k_2)k_2)}{2}\left[f_{+}(x_2,k_2)\frac{A^{*}f_{-}(x_2,k_2)}{(ik_2-\minf)}+f_{-}(x_2,k_2)\frac{A^{*}f_{+}(x_2,k_2)}{(-ik_2+\minf)}\right]\\
    &=\frac{i\nu(k_1,k_2)k_2-k_1\minf}{2(ik_2-\minf)}\left[f_{+}(x_2,k_2)f_{-}'(x_2,k_2) -f_{-}(x_2,k_2)f_{+}'(x_2,k_2)\right]\\
    &=\frac{i\nu(k_1,k_2)k_2-k_1\minf}{2(ik_2-\minf)}W_{H}(k_2) \, .
\end{align*}
\end{proof}

\subsection{Resolvent and Time Evolution Operator}\label{sec:2d_resolvent}
In this section we describe the time evolution associated with the two-dimensional operator $D$. We  decompose the operator $e^{-it\hatD}$ using the projections $P_{\rm edge}$ and $P_{\rm bulk}$ defined in Definition \ref{def:projections}. Similarly to subsection \ref{sec:1d_resolvent}, we express the bulk component of the time evolution operator in terms of the Jost functions $g_{\pm}(x,k_1,k_2)$ for $(k_1,k_2)\in\R^2$. Let $P_{\rm bulk}^+\equiv P_{\rm bulk}\chi (D>0)$ be the projection onto the positive part of the bulk subspace of $D$, and fix $\gamma >0$. Then for $x,y\in\R^2$
\begin{align*}
    [e^{-itD}\langle D\rangle^{-\gamma}P_{\rm bulk}^+](x,y) = \int_{\R} e^{ik_1 (x_1-y_1)}\left[e^{-it\hatD}\langle \hatD\rangle^{-\gamma}P_{\rm bulk}^+\right](x_2,y_2)\frac{dk_1}{2\pi} \,.
\end{align*}
For $k_1\in\R$ and $z\notin \sigma(\hatD)$, define the (fiber) resolvent $R_{\widehat{D}}(k_1,z)\equiv(\hatD-z)^{-1}$. Through a similar process to the one-dimensional case, see Sec.\ \ref{sec:1d_resolvent}, we find for $z=\nu(k_1,k_2)$
\begin{align*}
    R_{\widehat{D}}(k_1,\nu(k_1,k_2))(x_2,y_2) = \frac{i}{W_{\widehat{D}}(k_1,k_2)}\begin{cases}
        g_{+}(x_2,k_1,k_2)g_{-}(y_2,k_1,k_2)^{\top} & x_2 > y_2\\
        g_{-}(x_2,k_1,k_2)g_{+}(y_2,k_1,k_2)^{\top} & x_2 < y_2
    \end{cases}.
\end{align*}
The limiting resolvents $R_{\pm}(k_1,k_2)(x_2,y_2) = R_{\widehat{D}}(k_1,\nu(k_1,k_2\pm i0))(x_2,y_2)$ are given by
\begin{align*}
    R_{\pm}(k_1,k_2)(x_2,y_2) = \frac{i}{W_{\widehat{D}}(k_1,\pm k_2)}\begin{cases}
        g_{+}(x_2,k_1,\pm k_2)g_{-}(y_2,k_1,\pm k_2)^{\top} & x_2 > y_2\\
        g_{-}(x_2,k_1,\pm k_2)g_{+}(y_2,k_1,\pm k_2)^{\top} & x_2 < y_2
    \end{cases}.
\end{align*}
Then by Stone's formula, ignoring the negative part of the spectrum
\begin{align*}
    \left[e^{-it\hatD}\langle \hatD\rangle^{-\gamma}P_{\rm bulk}^{+}\right]&(x_2,y_2)\\
    &=\frac{1}{2\pi i}\int_{\sqrt{\minf^2 + k_1^2}}^{\infty} e^{-itz }\langle z\rangle ^{-\gamma}\left[R_{\widehat{D}}(k_1,z+i0)-R_{\widehat{D}}(k_1,z-i0)\right](x_2,y_2)dz.
    \end{align*}
Making the change of variables $z = \nu(k_1,k_2)$ in the integral above we arrive at the formula
\begin{multline*}
   \left[e^{-it\hatD}\langle \hatD\rangle^{-\gamma}P_{\rm bulk}^{+}\right](x_2,y_2) \\ = \frac{1}{2\pi i}\int_{0}^{\infty} e^{-it\nu(k_1,k_2)}\left[R_{+}(k_1,k_2)-R_{-}(k_1,k_2)\right](x_2,y_2)\frac{k_2 dk_2}{\nu(k_1,k_2)^{\gamma+1}},
\end{multline*}
where we once again have replaced $\langle \nu(k_1,k_2)\rangle$ by $\nu(k_1,k_2)$. Using the formulas for $R_{\pm}(k_1,k_2)$ above, and making the change $k_2\to-k_2$ in the second integral, we observe that the propagator kernel can be expressed in terms of the Jost functions $g_{\pm}(x_2,k_1,k_2)$: for $x_2>y_2$
\begin{multline}\label{eq:2d_propagator}
    [e^{-itD}\langle D\rangle^{-\gamma}P_{\rm bulk}^{+}](x,y) = \\ \frac{1}{(2\pi)^2}\int_{\R^2}e^{ik_1(x_1-y_1)}e^{-it\nu(k_1,k_2)}\frac{g_{+}(x_2,k_1,k_2)g_{-}(y_2,k_1,k_2)^{\top}}{W_{\widehat{D}}(k_1,k_2)}\frac{k_2 dk_1dk_2}{\nu(k_1,k_2)^{\gamma+1}},
\end{multline}
whereas for $x_2<y_2$
\begin{multline*}
    [e^{-itD}\langle D\rangle^{-\gamma}P_{\rm bulk}^{+}](x,y) =\\ \frac{1}{(2\pi)^2}\int_{\R^2}e^{ik_1(x_1-y_1)}e^{-it\nu(k_1,k_2)}\frac{g_{-}(x_2,k_1,k_2)g_{+}(y_2,k_1,k_2)^{\top}}{W_{\widehat{D}}(k_1,k_2)}\frac{k_2 dk_1dk_2}{\nu(k_1,k_2)^{\gamma+1}}.
\end{multline*}
Moreover, the restriction to the edge states is given by
\begin{align*}
    [e^{-itD}\langle D\rangle^{-\gamma}P_{\rm edge}](x,y)&=\frac{1}{2\pi}\sum_{j=-N}^{N}\int_{\R}e^{ik_1(x_1-y_1)}e^{-it\nu_j(k_1)}\Phi_{j }(k_1,x_2)\Phi_{j }(k_1,y_2)^{\top}\frac{dk_1}{ \nu_j(k_1)^{\gamma}},
\end{align*}
with $\nu_j(k_1)$ given in Lemma \ref{lem:hatD_spectrum}.

\subsection{Dispersive Decay Estimates}\label{sec:2d_proofs}

Having developed the spectral theory of the two-dimensional Dirac operator $D$, see \eqref{eq:intro-2d-dirac}, we now turn to prove the associated unweighted (Theorem \ref{thm:2d_main_theorem}) and weighted (Theorem \ref{thm:2dweighted}) dispersive estimates.

First we define an analog of the edge-resonance, which is suitable to these settings. This is the main obstruction to integrable rates of decay, and is once again determined by the properties of the Schr\"odinger Wronskian. Crucially, since each fiber operator $\widehat{D}(k_1)$ has a threshold resonance at a different energy, we {\em cannot} define in these settings a threshold resonant {\em energy.}

From Theorem \ref{thm:2d_wronskians}, we can see that the Dirac Wronskian $W_{\widehat{D}}(k_1,k_2)$ always vanishes at the origin $k_1=k_2=0$. Even so, since this is an isolated point in the $(k_1, k_2)$-plane,  it poses no obstruction for dispersion. However, for $k_1\neq 0$, $W_{\widehat{D}}(k_1,0)$ is zero or nonzero independently of the value of $k_1$. Thus, either {\em all} fibers have a resonance, or none has. We thus define:
\begin{definition}\label{def:2dReg}
  A Hamiltonian $D$ is said to be {\bf regular} if $W_{\widehat{D}}(k_1,0)\neq 0$ for all $k_1\neq 0$.
\end{definition}
We remark on the significance of this condition to the weighted estimate in Remark \ref{rmk:2dweighted}.
\begin{theorem}\label{thm:2d_main_theorem}
Assume that $m^2(x)-\minf^2,  m'(x) \in L^1_1$ in the regular case and $m^2(x)-\minf^2,  m'(x) \in L^1_2$ in the non-regular case. Let $P_{0},P_{\rm edge},$ and $P_{\rm bulk}$ be the projections defined in \eqref{eq:Pj2dDef}--\eqref{eq:Pbulk2d}. Then
\begin{align}
    \|e^{-itD}\langle D\rangle^{-\gamma}(P_{\rm edge}-P_0)\|_{L^1\to L^{\infty}}&\lesssim \la t\ra^{-1/2}  \, , \qquad &\gamma > 3/2 \, ,  \label{eq:disp_pzero}\\
    \|e^{-itD}\langle D\rangle^{-\gamma}P_{ \rm bulk}\|_{L^1\to L^{\infty}}&\lesssim \la t\ra^{-1},\qquad &\gamma >2 \, .\label{eq:disp_pbulk}
\end{align}
\end{theorem}
\begin{proof}
For the edge part, \eqref{eq:disp_pzero}, fix $1\leq j\leq N$ (the negative branches are treated similarly) and consider the propagation of a single edge mode, $[e^{-itD}\langle D\rangle^{-\gamma}P_j](x,y)$. By Definition \ref{def:projections} we have
\begin{align*}
    [e^{-itD}\langle D\rangle^{-\gamma}P_j](x,y) &=\int_{\R}e^{ik_1(x_1-y_1)}e^{-it\nu_j(k_1)}\Phi_{j }(k_1,x_2)\Phi_{j }(k_1,y_2)^{\top}\frac{dk_1}{\la \nu_j(k_1)\ra^{\gamma}},
\end{align*}
which, by the definition of $\Phi_j(k_1,x_2)$ given in Lemma~\ref{lem:hatD_spectrum}, can be expressed as a $j$-dependent linear combination of integrals of the form
\begin{align*}
    \int_{\R}e^{ik_1(x_1-y_1)}e^{-it\nu_j(k_1)}\frac{(\nu_j(k_1)+k_1)^{n}}{\nu_j(k_1)}\frac{dk_1}{\la \nu_j(k_1)\ra^{\gamma}},\qquad n = -1,0,1,
\end{align*}
times $k_1$-independent functions of $x_2$ (the eigenfunctions of $H$ and $\widetilde{H}$), which are in $L^{\infty}(\R_{x_2})\cap L^2(\R_{x_2})$. As $\gamma > 3/2$, the integral above is bounded by $C_j\la t\ra^{-1/2}$ by an application of Lemma~\ref{lem:KGdisp} along with the fact that $\frac{(\nu_j(k_1)+k_1)^{n}}{\nu_j(k_1)^{1+\eps}}\in A(\R)$ for any $\eps>0$ when $n\leq 1$ by Lemma \ref{lem:hhat}. This completes the proof of the edge estimate, \eqref{eq:disp_pzero}.\\

Turning to the bulk estimate, \eqref{eq:disp_pbulk}: up to constants, and WLOG for $x_2<y_2$,
\begin{align*}
      [e^{-itD}\langle D\rangle^{-\gamma}P_{\rm bulk}^{+}](x,y)&=\int_{\R^2}e^{ik_1(x_1-y_1)}e^{-it\nu(k_1,k_2)}\frac{g_{-}(x_2,k_1,k_2)g_{+}(y_2,k_1,k_2)^{\top}}{W_{\widehat{D}}(k_1,k_2)}\frac{k_2 dk_1dk_2}{\nu(k_1,k_2)^{\gamma+1}}.
\end{align*}
Now fix $\eps >0$ and define the function $F(x_2,y_2,k_1,k_2)$ as 
\begin{align*}
    F(x_2,y_2,k_1,k_2) = \frac{k_2 g_{-}(x_2,k_1,k_2)g_{+}(y_2,k_1,k_2)}{W_{\widehat{D}}(k_1,k_2)\nu(k_1,k_2)^{1+\eps}}.
\end{align*}
To complete the dispersive estimate \eqref{eq:disp_pbulk}, we use a similar argument to the one employed in Sec.\ \ref{sec:1d_estimates}:  let $\gamma > 2$ and choose $\eps>0$ so that $\gamma -\eps>2$. Using Fourier multiplication, for $x_2<y_2$, we can bound \eqref{eq:2d_propagator} by  
\begin{multline}\label{eq:FourierMult2d}
    \|e^{-itD}\langle D\rangle^{-\gamma}P_{\rm bulk}^+\|_{L^1\to L^{\infty}}\leq \\ \sup\limits_{x_1,y_1}\Big\| \mathcal F^{-1}_{k\to \xi}\Big(\frac{e^{ik_1(x_1-y_1)}e^{-it\nu(k_1,k_2)} }{\nu(k_1,k_2)^{\gamma-\epsilon} } \Big) \Big\|_{L^\infty_{\xi}} \sup\limits_{x_2,y_2}\Big\|F(x_2,y_2,k_1,k_2)  \Big\|_{A(\R^2_k)}.
\end{multline}

The first factor is just the decay by the Klein-Gordon dispersion relation, and so by  Lemma~\ref{lem:KGdisp}, and noting that the oscillatory factors in the first term do not alter the $L^{\infty}$ norm of the Fourier transform, we obtain the bound $\la t\ra^{-1}$ provided that $\gamma-\epsilon>2$  and $\epsilon>0$. All that remains is to prove that $F(x_2,y_2,\cdot ,\cdot)$ is in $A(\R^2)$ for every fixed $(x_2,y_2)$ (i.e., the Fourier transform in $k_1,k_2$ is in $L^1$), and that $\|F(x_2,y_2,\cdot ,\cdot)\|_{A(\R^2)}$ is uniformly bounded in $(x_2,y_2)$.

To see that, first note that by Theorem~\ref{thm:2d_wronskians} we have
\begin{align*}
     W_{\widehat{D}}(k_1,k_2) &=\frac{i\nu(k_1,k_2)k_2-k_1\minf}{2(ik_2-\minf)}W_{H}(k_2),
\end{align*}
where $W_{H}(k_2)=W[f_{+}(\cdot,k_2),f_{-}(\cdot,k_2)]$ is the Schr\"odinger Wronskian associated to $H$. In turn, $F(x_2,y_2,k_1,k_2)$ may be expressed
\begin{align*}
     F(x_2,y_2,k_1,k_2) = \frac{2k_2(ik_2-\minf) g_{-}(x_2,k_1,k_2)g_{+}(y_2,k_1,k_2)}{W_{H}(k_2)(i\nu(k_1,k_2)k_2-k_1\minf)\nu(k_1,k_2)^{1+\eps}}.
\end{align*}
Using \eqref{eq:gpmdef}, this can, up to constants, be expressed as the product 
\begin{align*}
    F(x_2,y_2,k_1,k_2) = e^{-ik_2(y_2-x_2)}\frac{k_2\varphi_{-}(x_2,k_2)\varphi_{+}(y_2,k_2)}{W_H(k_2)}h(k_1,k_2)\, ,
\end{align*}
where the functions $h(k_1,k_2)$ are of the form

\begin{align*}
   h(k_1,k_2)= \frac{ (ik_2-\minf) [\nu(k_1,k_2)+ m_{\infty} \pm (k_1+ik_2)] [\nu(k_1,k_2)- m_{\infty}\pm (k_1-ik_2)]  }{(i\nu(k_1,k_2)k_2-k_1\minf)\nu(k_1,k_2)^{1+\eps}}.
\end{align*}
In the proof of Theorem~\ref{thm:1d_unweighted}, we showed that, uniformly in $x_2,y_2$, $$\frac{k_2\varphi_{-}(x_2,k_2)\varphi_{+}(y_2,k_2)}{W_H(k_2)}\in \mF TV(\R_{k_2}).
$$
Therefore, it suffices to prove that $h\in A(\R^2)$. Note that 
\begin{equation}\label{eq:nu-m}\nu(k_1,k_2)- m_{\infty} =\frac{k_1^2+k_2^2}{\nu(k_1,k_2)+m_{\infty}}.  
\end{equation}
 Therefore, $h$ is the product of 
 \begin{align*}
& \frac{ (ik_2-\minf) (k_1-ik_2)}{(i\nu(k_1,k_2)k_2-k_1\minf) \nu(k_1,k_2)^{ \frac{\eps}2}} \, \text{ and } \, \\ &\frac{[\nu(k_1,k_2)+ m_{\infty} \pm (k_1+ik_2)] [k_1+ik_2\pm (\nu(k_1,k_2)+m_{\infty})]  }{(\nu(k_1,k_2)+\minf) \nu(k_1,k_2)^{1+\frac{\eps}2}}.
\end{align*}
The second factor is in $A(\R^2)$ by Lemma~\ref{lem:hhat} below. Note that by direct computation and \eqref{eq:nu-m}, we have 
$$
 \frac{1}{i\nu(k_1,k_2)k_2-k_1\minf}= \frac{-i\nu(k_1,k_2)k_2-k_1\minf}{(k_1^2+k_2^2)(\minf^2+k_2^2)}=\frac{-\minf(k_1+ik_2)-i(\nu(k_1,k_2)-\minf)k_2}{(k_1^2+k_2^2)(\minf^2+k_2^2)}
 $$
$$
=-\frac{\minf(k_1+ik_2)}{(k_1^2+k_2^2)(\minf^2+k_2^2)} -i \frac{k_2}{(\minf^2+k_2^2)(\nu(k_1,k_2)+\minf) }
$$
Therefore, we can rewrite the first factor in $h$ as  
$$
  \frac{ -\minf(ik_2-\minf) }{ (\minf^2+k_2^2) \nu(k_1,k_2)^{ \frac{\eps}2} }-i \frac{k_2 (ik_2-\minf)}{(\minf^2+k_2^2)}\frac{ (k_1-ik_2)}{(\nu(k_1,k_2)+\minf)  \nu(k_1,k_2)^{ \frac{\eps}2}}.
$$ 
Using Lemma~\ref{lem:hhat}, we see that the first summand is in $A(\R^2)$, the first factor of the second summand is in $\mF TV(\R^2)$, and the second factor is in $A(\R^2)$.   This finishes the proof of $h\in A(\R^2)$. 
\end{proof}

For regular cases, in the sense of Definition \ref{def:2dReg}, we have the following improved weighted estimate:
\begin{theorem}\label{thm:2dweighted}
Assume that $m^2(x)-\minf^2,  m'(x) \in L^1_3$ and that $D$ is regular. Then   for  $\gamma> 3$
\begin{align*}
    \|\la x_2\ra^{-1} e^{-itD}\langle D\rangle^{-\gamma}P_{\rm bulk}(x,y) \la y_2\ra^{-1}\|_{L^1\to L^\infty}\les  \langle t\rangle ^{-2} \, .
\end{align*}
\end{theorem}

\begin{rmk}\label{rmk:2dweighted}
The additional factor of $t^{-1}$ in the weighted estimate, Theorem~\ref{thm:2dweighted}, is notable; indeed, in the case of the two-dimensional massive Dirac equation ($m\equiv {\rm const}$) with a   short-range potential, the rate improvement is only logarithmic, leading to a $t^{-1}(\log t)^{-2}$ decay \cite{erdougan2018dispersive}. This latter improvement is the result of a logarithmic singularity of the
two-dimensional free Schr{\"o}dinger resolvent; see
\cite[Theorem~1.3 and Section~4.1]{erdougan2018dispersive} for details.     

In the domain wall settings, the gain is due to a different mechanism, the fiber decomposition. As noted, by the explicit form of the (two-dimensional) Dirac Wronskian, \eqref{eq:2d_wroskian}, either all or none of the fiber operators, $\widehat{D}(k_1)$, have a threshold resonance. When that resonance occur, for each $k_1$ we get a 
non-decaying generalized eigenfunction of $\widehat{D}(k_1)$ at $\nu(k_1, 0)$ energy. Superposition over $k_1$ therefore yields
generalized wave packets which remain non-localized in $x_2$ and
disperse only along the interface, thus obstructing faster weighted decay (and are not influenced by $\langle x_2 \rangle ^{-1}$ weights). Technically, this obstruction
appears as a $k_2^{-1}$ singularity in the fiber resolvent amplitude,
which cancels the factor $k_2$ arising from the change of variables
$d\lambda=k_2\nu(k_1,k_2)^{-1}\,dk_2$; see \eqref{eq:weighted_2dPro}. In the absence of such a resonance, we integrate by parts in the $k_2$-variable, yielding a factor of $t^{-1}$, whereby
the remaining two-dimensional Klein-Gordon integral gives the usual
$t^{-1}$ decay, as in the unweighted estimate (Theorem~\ref{thm:2d_main_theorem}).
\end{rmk}

As noted in the introduction, the proof in the two-dimensional case follows the strategy of the weighted one-dimensional estimate, Theorem \ref{thm:1dWeighted}, with some key distinctions.
\begin{proof}
Recall from the proof of the previous theorem that up to constants, for $x_2<y_2$
\begin{multline}\label{eq:weighted_2dPro}
      [e^{-itD}\langle D\rangle^{-\gamma}P_{\rm bulk}^{+}](x,y) = \\ \int_{\R^2}e^{ik_1(x_1-y_1)} \frac{k_2 e^{-it\nu(k_1,k_2)}}{\nu(k_1,k_2) } \frac{g_{-}(x_2,k_1,k_2)g_{+}(y_2,k_1,k_2)^{\top}}{W_{\widehat{D}}(k_1,k_2)\nu(k_1,k_2)^{\gamma }}  dk_1dk_2.
\end{multline}
Writing $e^{-it\nu(k_1,k_2)}\frac{k_2}{\nu(k_1,k_2)}=\frac{i}{t}\partial_{k_2} e^{-it\nu(k_1,k_2)}$ and integrating by parts in the $k_2$ variable, we obtain 
\begin{multline}\label{eq:prop_preFourier_2dw}
  [e^{-itD}\langle D\rangle^{-\gamma}P_{\rm bulk}^{+}](x,y) = \\ \frac{1}{  i t }\int_{\R^2}e^{ik_1(x_1-y_1)}   e^{-it\nu(k_1,k_2)}\partial_{k_2}\Big(\frac{g_{-}(x_2,k_1,k_2)g_{+}(y_2,k_1,k_2)^{\top}}{W_{\widehat{D}}(k_1,k_2)\nu(k_1,k_2)^{\gamma }}\Big)  dk_1dk_2. 
\end{multline}
 As in the proof of Theorem \ref{thm:2d_main_theorem} (in particular, \eqref{eq:FourierMult2d}), we apply the Fourier multiplication formula to \eqref{eq:prop_preFourier_2dw}. Here, we get the extra $t^{-1}$ decay from the Klein-Gordon kernel estimates (Sec.\ \ref{sec:KG}), and  the theorem follows from the following claim:
\begin{equation}\label{eq:2dweightclaim}
 \Big\|\nu(k_1,k_2)^{2+} \partial_{k_2}\Big(\frac{g_{-}(x_2,k_1,k_2)g_{+}(y_2,k_1,k_2)^{\top}}{W_{\widehat{D}}(k_1,k_2)\nu(k_1,k_2)^{\gamma }}\Big)\Big\|_{A(\R^2)}\les \la x_2\ra\la y_2\ra.
 \end{equation} 
 We can rewrite the function above as 
\begin{align*}
   \nu(k_1,k_2)^{2+} \partial_{k_2}\Big(e^{-ik_2(y_2-x_2)}\varphi_{-}(x_2,k_2)\varphi_{+}(y_2,k_2)\frac{ (ik_2-\minf)}{W_H(k_2)}h(k_1,k_2)\Big),
\end{align*}
where the functions $h(k_1,k_2)$ are of the form (by a slight modification of the calculations in the proof of Theorem~\ref{thm:2d_main_theorem})
\begin{multline*}
   h(k_1,k_2)= \frac{  [\nu(k_1,k_2)+ m_{\infty} \pm (k_1+ik_2)] [\nu(k_1,k_2)- m_{\infty}\pm (k_1-ik_2)]  }{(i\nu(k_1,k_2)k_2-k_1\minf)\nu(k_1,k_2)^{\gamma }}\\=\frac{   (k_1-ik_2)}{(i\nu(k_1,k_2)k_2-k_1\minf)  }    \frac{[\nu(k_1,k_2)+ m_{\infty} \pm (k_1+ik_2)] [k_1+ik_2\pm (\nu(k_1,k_2)+m_{\infty})]  }{(\nu(k_1,k_2)+\minf) \nu(k_1,k_2)^{\gamma }} \\ = \Big(\frac{ -\minf  }{ (\minf^2+k_2^2)   }-i \frac{k_2  }{(\minf^2+k_2^2)}\frac{ (k_1-ik_2)}{(\nu(k_1,k_2)+\minf)   }\Big) \times\\ \frac{[\nu(k_1,k_2)+ m_{\infty} \pm (k_1+ik_2)] [k_1+ik_2\pm (\nu(k_1,k_2)+m_{\infty})]  }{(\nu(k_1,k_2)+\minf) \nu(k_1,k_2)^{\gamma }},
\end{multline*} 
which shows that $h(k_1,k_2)$ is well-behaved near the origin. Now, the claim follows from similar estimates to the one-dimensional weighted estimates, \eqref{Wcl1} and \eqref{Wcl2}, and the following implications of Lemma~\ref{lem:hhat}:
$$\big\|\nu(k_1,k_2)^{2+}   h(k_1,k_2) \big\|_{A(\R^2)} \les 1 \, , \qquad \big\|\nu(k_1,k_2)^{2+} \partial_{k_2}  h(k_1,k_2) \big\|_{A(\R^2)} \les 1
$$
provided that   $\gamma>3$.
\end{proof}

\section{Dispersive Estimates for two-dimensional Dirac operators with an unbounded domain wall}\label{sec:unbd}

In this section we obtain dispersive estimates for the two-dimensional operator
\begin{align*}
    D = -i\partial_{x_1}\sigma_1-i\partial_{x_2}\sigma_2+m(x_2)\sigma_3,
\end{align*}
with an {\em unbounded} domain wall $m$, i.e.,  $m\in C^1(\R)$, $\displaystyle\lim_{x_2\to\pm\infty}m(x_2)=\pm \infty$, and
$$\lim_{x_2\to\pm\infty}(m^2(x_2)\pm m'(x_2))= \infty \, .$$
It may seem natural to decompose $D$ into its fibers $\widehat D(k_1)$ (as we did in the bounded case in Sec.~\ref{sec:2d_bounded_estimates}), and then decompose each fiber into its $x_2$ eigenbasis. In our experience, it is better to use an eigenbasis in $x_2$, which keeps the $x_1$-one-dimensional Schr{\"o}dinger structure of each eigenfunction. Concretely, for $z\in\C\setminus\R$ we have
\begin{subequations}\label{eq:dirac_res}
\begin{align}
(D-z)(D+z)=D^2-z^2 =  (-\partial_{x_1}^2-z^2)I +H_{x_2} \, ,
\end{align}
where 
\begin{equation}
    H_{x_2}\equiv \begin{pmatrix} -\partial_{x_2}^2+m^2(x_2) & m'(x_2)\\ m'(x_2) & -\partial_{x_2}^2+m^2(x_2) \end{pmatrix} \, .
\end{equation}
\end{subequations}
We note that 
$$ S H_{x_2} S = \begin{pmatrix}
    H_+ &0 \\ 0 & H_-
\end{pmatrix} , \quad {\rm where} \quad S=\frac{1}{\sqrt{2}}\begin{pmatrix}
    1 &1 \\ 1 & -1
\end{pmatrix} , \quad  H_{\pm} = -\partial_{x_2}^2+m^2(x_2) \pm m'(x_2) .$$
Note that $ H_{+} = AA^*$ and $H_{-} = A^*A$ where the raising and lowering operators are defined as in \eqref{def:AA*}. These one-dimensional Schr\"odinger Hamiltonians are analogous to those defined in Sec.\ \ref{sec:SUSY}; however, they are non-negative with trapping potentials $V_{\pm}(x_2)\equiv m^2(x_2)\pm m'(x_2)$ which are asymptotically infinite as $x_2\to\pm\infty$. As each of $H_{\pm}$ has compact resolvent, and since we have the relation
$$\sigma(H_{-})=\sigma(H_{+})\cup\{0\},\quad 0\notin \sigma(H_+) \, ,$$
we have the following lemma which describes the spectrum of $H_{x_2}$. Its proof follows from the properties of $H_{\pm}$, and a direct calculation using the diagonal form of $SH_{x_2}S$ above.
\begin{lemma}\label{lem:Hx2}
    The operator $H_{x_2}$ has compact resolvent with eigenvalues given by $\sigma(H_{-})$. 
     Denote by $H_{-}\phi_{n}=\lambda_{n}\phi_{n}$ the normalized eigenpairs of $H_{-}$ with the convention that $\lambda_{0} = 0$, and $\pi_n = \dfrac{1}{\sqrt{\lambda_n}}A\phi_n$ for $n\geq 1$. The corresponding eigenspaces of $H_{x_2}$ are therefore
    \begin{align*}
        E_{0} &\equiv  \ker(H_{x_2}) = span\left\{ \Phi_0 \right\} \,,\\
        E_{n} &\equiv\ker(H_{x_2}-\lambda_n)= span\left\{ \Phi_n (x_2), \Psi _n (x_2)\right\} \, ,\qquad n\in \mathbb{N} \, ,
    \end{align*}
where we used the shorthand
\begin{equation}
    \Phi_n (x_2) \equiv \frac{1}{\sqrt{2}}\begin{pmatrix}
            \phi_n (x_2) \\ -\phi_n(x_2)
        \end{pmatrix} \,  , \qquad \Psi_n (x_2) \equiv \frac{1}{\sqrt{2}}\begin{pmatrix}
            \pi_n (x_2) \\ \pi_n(x_2)
        \end{pmatrix}  \, .
\end{equation}
\end{lemma}
Lemma \ref{lem:Hx2} allows us to expand the resolvent of the Dirac operator, $D$, in terms of the eigenfunctions $\Phi_n$ and $\Psi_n$.
Any $f\in L^2(\R^2)$ can be written as 
$$ f(x_1,x_2) = b_0 (x_1)\Phi_0 (x_2) + \sum\limits_{n\geq 1}\left(a_n (x_1) \Psi_n (x_2) +b_n (x_1)\Phi_n (x_2)\right) , $$
 where $b_n (x_1) =\langle \Phi_n, f(x_1,\cdot)\rangle_{L^2 (\R_{x_2} )}$, and similarly for the $a_n$'s. 
  Now, for ease of notation, let us reorganize 
$$ f_{2n-1} = a_n \, , \qquad f_{2n}= b_n\,  , \qquad \Gamma_{2n-1} = \Psi_n  \, , \qquad \Gamma_{2n}=\Phi_n  \, .$$
Then the eigenvalue corresponding to $\Gamma_k$, $k=0,1,\ldots ,$ is $\mu_k \equiv \lambda_{\lceil k/2 \rceil }$.

The mixed Lebesgue norms $L_{x_1}^p L_{x_2}^2$ may be computed as
$$
\|f\|_{L^p_{x_1} L^2_{x_2}}=  \Big\|\Big(\sum_{k=0}^\infty |f_k(x_1) |^2\Big)^{1/2}\Big\|_{L^p_{x_1}} \, ,
$$ 
by Parseval's Theorem in $L^2_{x_2}$. Hence, by \eqref{eq:dirac_res} and Lemma~\ref{lem:Hx2}
\begin{align*}
    [(D^2-z^2)f](x_1,x_2) &= \sum\limits_{k=0}^{\infty} \Gamma_k(x_2)  \left( -\partial_{x_1}^2-z^2 + \mu_k\right)f_k(x_1)  .
    \end{align*}
    Thus, denoting the free Schr{\"o}dinger resolvent $R_0 (E)\equiv (-\partial_{x_1}^2-E)^{-1}$, we have that
    \begin{align*}
        (D^2-z^2)^{-1} f &= \sum\limits_{k=0}^{\infty} \Gamma_k R_0 (z^2-\mu_k)f_k  \, , \qquad \qquad {\rm and}~{\rm so} \\
        (D-z)^{-1}f &= (D+z)(D^2-z^2)^{-1} f=\left( D+z\right)\sum\limits_{k=0}^{\infty} \Gamma_k R_0 (z^2-\mu_k)f_k \\
        &=  \sum\limits_{k=0}^{\infty} \left( -i\partial_{x_2} \sigma_2 + m(x_2)\sigma_3\right)\Gamma _k \left[ R_0 (z^2-\mu_k)f_k\right] + \sum\limits_{k=0}^{\infty} (-i\partial_{x_1} \sigma _1 +z)\Gamma_k R_0 (z^2-\mu_k )f_k \\ 
        &= \sum\limits_{k=1}^{\infty} \mu_k ^{1/2}\Gamma _{k-(-1)^k} \left[ R_0 (z^2-\mu_k)f_k\right] +  \sum\limits_{k=0}^{\infty} \Gamma_k( i(-1)^k\partial_{x_1}  +z)R_0 (z^2-\mu_k )f_k \numberthis \label{eq:2dRaisingTrick} \\ &\equiv
        R_{D,1}(z)f + R_{D,2}(z)f,\numberthis \label{eq:unbdm_res}
    \end{align*}
where, in deriving \eqref{eq:2dRaisingTrick}, we simplify $R_{D,1}$ using the coordinate representation of $\Phi_n$ and $\Psi_n$. For $k$ even, $\Gamma_k= \Phi_{k/2}$, and so
\begin{align*}
    \left[ -i\partial_{x_2} \sigma_2+m(x_2)\sigma _3\right]\Phi_n (x_2) &= \frac{1}{\sqrt{2}}\begin{pmatrix}
     (m+\partial_{x_2})\phi_n \\ (m+\partial_{x_2})\phi_n
    \end{pmatrix}  
     = \frac{1}{\sqrt{2}}\begin{pmatrix}
        A\phi_n \\ A\phi_n
    \end{pmatrix}  
    =\lambda_n^{1/2} \Psi_n=\mu_k^{1/2} \Gamma_{k-1},
\end{align*}
whereas for $k$ odd $\Gamma_k=\Psi_{\frac{k+1}2}$, and so 
\begin{align*}
    \left[ -i\partial_{x_2} \sigma_2+m(x_2)\sigma _3\right]\Psi_n (x_2) &= \frac{1}{\sqrt{2}}\begin{pmatrix}
        A^{*}\pi _n \\ - A^{*}\pi_n
    \end{pmatrix} = \lambda_n^{1/2} \Phi_n(x_2) =\mu_k^{1/2} \Gamma_{k+1}.
\end{align*}
Similarly, we used $\sigma_1\Gamma_k=(-1)^{k+1}\Gamma_k$ in $R_{D,2}.$

\begin{rmk}\label{rmk:limiting_absorb} 
Equation \eqref{eq:unbdm_res} gives us the limiting absorption principle and allows us to define the limiting resolvent $(D-\lambda\pm i0)^{-1}$ required to apply Stone's theorem.
\end{rmk}
The following theorem states the dispersive estimates for general unbounded $m$:  
\begin{theorem}\label{thm:unbndmdisp} Suppose that $m$ satisfies the assumptions at the beginning of the section. Let $P_0$ be the projection $P_0f (x_1,x_2)\equiv f_0(x_1) \Gamma_0(x_2)$. Then for any $\gamma>  \frac32 $, we have 
\begin{equation}\label{eq:disp_unb}
\|e^{-itD}\la D\ra^{-\gamma}  (I-P_0)\|_{L^1_{x_1} L^2_{x_2} \to L^\infty_{x_1} L^2_{x_2}} \les \la t\ra^{-1/2} \, .
\end{equation}
In addition, we have the following weighted bound corresponding to the edge mode. For any $\sigma\geq 0$, $\gamma>1$, and $p\in [1,\infty]$,
\begin{equation}\label{eq:P0est}
    \|\la x_1\ra^{-\sigma}e^{-itD} \la D\ra^{-\gamma} P_0f \|_{L^\infty_{x_1 }L^p_{x_2}}\les \la t\ra^{-\sigma} \| f_0\|_{L^1_{\sigma}} \, .
\end{equation}
\end{theorem}

\begin{rmk}
\begin{enumerate}
    \item Weights do not seem to improve the rate of decay in the edge branches in $I-P_0$. The reason is that each such branch can be mapped to a one-dimensional massive Dirac equation, which has a threshold resonance (if no potential is added); see e.g., \cite{erdougan2021one}.
    \item The weighted estimate for $P_0$, \eqref{eq:P0est}, may be understood by its real space representation: indeed, the dispersion relation in ${\rm image}(P_0)$ is simply $-k_1$, i.e., waves that travel to the left and do not disperse: $[e^{-iDt}P_0 f](x_1,x_2)= [P_0 f](x_1+t, x_2)$. Therefore, since  the profile is fixed in the traverse direction, $x_2$, and smoothing $\la D \ra^{-\gamma}$  guarantees that the data is in $L^{\infty}(\R ^2)$, the weighted decay is simply the result of the wave traveling away from the origin, thus the overall amplitude is attenuated by the weights.
\end{enumerate}
\end{rmk}
\begin{proof}
To obtain the unweighted bound, it suffices to consider the positive part of the spectrum. Writing $e^{-itD}\la D\ra^{-\gamma}(I-P_0)$ using Stone's formula, we obtain
\begin{align*}
    e^{-itD}\la D\ra^{-\gamma}(I-P_0) = \frac{1}{2\pi i}\int_{0}^{\infty}e^{-it\lambda }\la\lambda\ra^{-\gamma} (I-P_0)\left[(D-(\lambda+i0))^{-1}-(D-(\lambda-i0))^{-1}\right]d\lambda \, .
\end{align*}
Using \eqref{eq:unbdm_res} and Remark~\ref{rmk:limiting_absorb}, the contribution of $R_{D,1}$    is (note that the $n=0$ term does not appear due to the projection $I-P_0$):
$$
    \int_{\sqrt{\mu_1}}^\infty e^{-it\lambda }\la\lambda\ra^{-\gamma} \sum_{n=1}^\infty \sqrt{\mu_n}\Gamma_{n-(-1)^n}(x_2)  \chi_{  [\sqrt{\mu_n},\infty)}(\lambda)  [R_0^+(\lambda^2-\mu_n)-R_0^-(\lambda^2-\mu_n)] f_n (x_1) d\lambda \, .
$$ 
Note that the difference of limiting Schr\"odinger resolvents in the formula above vanishes for $\mu_n>\lambda^2$, and hence the $\lambda$ integral starts at $\sqrt{\mu_1}$. Using the formula for the one-dimensional free Schr\"odinger resolvent,
\begin{align*}
    R_0^{\pm}(s^2)f(x) = \frac{\pm i}{2s}\int_{\R} e^{\pm is|x-y|}f(y)dy \, ,
\end{align*}
and Fubini's theorem,\footnote{Note that each $f_n\in L^1$ and we can restrict the $\lambda$ integral to $[\sqrt{\mu_1},L]$ for arbitrary $L$, which turns the sum in $n$ into a finite sum.} we rewrite the   integral as a linear combination  of 
$$
 \sum_{n=1}^\infty  \Gamma_{n-(-1)^n}(x_2)    \int_\R   f_n (y)   \sqrt{\mu_n} \int_{\sqrt{\mu_n}}^\infty      e^{-it\lambda\pm i\sqrt{\lambda^2-\mu_n} |x_1-y|} \frac{ d\lambda}{\la\lambda\ra^{\gamma}\sqrt{\lambda^2-\mu_n}} \  dy.
$$ 
Since the $\Gamma_n$'s are $L^2_{x_2}$-orthonormal, we can bound the $ L^\infty_{x_1}L^2_{x_2}$ norm of the last expression  by 
\begin{equation}\label{eq:implicit}
\sup_{x_1}  \Big\|  \int_\R   f_n(y) C_{n,y,x_1,t}  dy  \Big\|_{\ell^2_n},
\end{equation}
  where 
  $$
   C_{n,y,x_1,t}  \equiv  \sqrt{\mu_n}  \int_{\sqrt{\mu_n}}^\infty   e^{-it\lambda\pm i\sqrt{\lambda^2-\mu_n} |x_1-y|}  \frac{ d\lambda}{\la\lambda\ra^{\gamma}\sqrt{\lambda^2-\mu_n}}.
  $$
By Minkowsky's integral inequality, we have 
$$
\eqref{eq:implicit} \les  \int_\R  \| f_n(y)\|_{\ell^2_n} \, dy ~\cdot ~ \sup_{x_1,y,n} |C_{n,y,x_1,t}|= \|f\|_{L^1_{x_1}L^2_{x_2}} \sup_{x_1,y,n} |C_{n,y,x_1,t}|  .
$$
By an analogous calculation for the contribution of $R_{D,2}$ in  \eqref{eq:unbdm_res}, the bound \eqref{eq:disp_unb}   follows from the following claim:
\begin{equation}\label{eq:CnDnbd}
    \sup_{x_1,y,n} \big( |C_{n,y,x_1,t}|  +  |D_{n,y,x_1,t}| \big) \les \la t\ra^{-1/2},
\end{equation}
where 
 $$
D_{n,y,x_1,t} \equiv \int_{\sqrt{\mu_n}}^\infty  e^{-it\lambda\pm i\sqrt{\lambda^2-\mu_n} |x_1-y|}  \Big(\mp (-1)^n \sgn(x_1-y )+\tfrac{\lambda}{\sqrt{\lambda^2-\mu_n}}  \Big)  \frac{d\lambda}{\la\lambda\ra^{\gamma} }.
 $$
 To prove the claim \eqref{eq:CnDnbd}, introduce the change of variables 
 $$\lambda=\la s\ra\sqrt{\mu_n}, \quad  d\lambda=\sqrt{\mu_n}  \frac{s}{\la s\ra}ds, \quad \sqrt{\lambda^2-\mu_n}= s\sqrt{\mu_n},$$
and rewrite 
 \begin{equation}\label{eq:Cnrewrite}
 C_{n,y,x_1,t}= \sqrt{\mu_n} \int_0^\infty e^{-it\sqrt{\mu_n} \la s\ra\pm i s \sqrt{ \mu_n} |x_1-y|}     \frac{ds}{\la s\ra \la  \la s\ra \sqrt{\mu_n} \ra^{ \gamma} }.
 \end{equation}
 For $\gamma >3/2$, a slight variant of Lemma~\ref{lem:KGdisp} below (note that $\la \la s\ra \sqrt{\mu_n}\ra$ behaves like $  \la s\ra\sqrt{\mu_n}$),  yields 
 $$
 |C_{n,y,x_1,t}|\les \mu_n^{\frac12-\frac\gamma2} \la t\sqrt{\mu_n}\ra^{-\frac12}.  
 $$
 For $|t|>1$, we can bound this by  $\mu_n^{\frac14-\frac\gamma2} | t|^{-\frac12}$, which is $\lesssim t^{-1/2}$ already when $\gamma>\frac12$.   For $|t|<1$,  by taking absolute values in \eqref{eq:Cnrewrite} and ignoring oscillations, absolute summability implies the result already when  $\gamma > 1$.
 
 By the same change of variable, we rewrite
  $$
D_{n,y,x_1,t}  = \int_0^\infty e^{-it\sqrt{\mu_n} \la s\ra\pm i s \sqrt{ \mu_n} |x_1-y|}   \Big(\mp (-1)^n \sgn(x_1-y )\frac{s}{\la s\ra}+1 \Big)     \frac{ \sqrt{\mu_n} ds}{  \la  \la s\ra \sqrt{\mu_n} \ra^{ \gamma} }.
 $$
 Once again, by Lemma~\ref{lem:KGdisp}, if $\gamma>\frac32$, we obtain 
 $$
 |D_{n,y,x_1,t}|\les \mu_n^{\frac12-\frac\gamma2} \la t\sqrt{\mu_n}\ra^{-\frac12}\les \la t\ra^{-\frac12}.
 $$
This completes the proof of \eqref{eq:CnDnbd}, and therefore of the unweighted bound.  

For the weighted bound it is more convenient to consider the full spectrum, $\R$, as there is no gap. By Stone's formula and \eqref{eq:unbdm_res}, WLOG for $\lambda>0$,
\begin{align*}
e^{-itD}\la D\ra^{-\gamma} P_0 f=  \Gamma_0(x_2)   &\int_{\R} \la\lambda\ra^{-\gamma} e^{-it\lambda}  (i \partial_{x_1}+\lambda) [R_0^+(\lambda^2)-R_0^-(\lambda^2)] f_0(x_1) d\lambda \\
&= \Gamma_0(x_2)  \int_\R \int_\R \la\lambda\ra^{-\gamma}  e^{-it\lambda\pm i\lambda |x_1-y_1|}  (\sgn(x_1-y_1)\pm 1)  f_0(y_1) d\lambda dy_1 \\
&=
\Gamma_0(x_2) \int_\R \mF(\la \cdot\ra^{-\gamma})(t\pm|x_1-y_1|)    (\sgn(x_1-y_1)\pm 1)  f_0(y_1)   dy_1. 
\end{align*}
As $\gamma>1$, it is a well-known property of the Bessel kernel that for any $\sigma\geq 0$ we have 
\begin{align*}
    | \mF(\la \cdot\ra^{-\gamma})(t\pm|x_1-y_1|)|&\les \la t\pm|x_1-y_1|\ra^{-\sigma} \les \frac{\la x_1\ra^\sigma\la y_1\ra^\sigma}{\la t\ra^\sigma} \, .
\end{align*}
This yields the claim.
\end{proof}
\begin{rmk}  The proof of Theorem~\ref{thm:unbndmdisp} yields the following stronger bound (see \eqref{eq:implicit})
 $$
 \|e^{-itD}\la D\ra^{-\gamma} (I-P_0)f\|_{  L^\infty_{x_1} L^2_{x_2}} \les \la t\ra^{-\frac12}\|f_n(y)\|_{\ell^2_nL^1_y}\, ,
 $$
 which is a weaker norm than  the $L^1_{x_1}L^2_ {x_2}=L^1_{x_1}\ell^2_n$ norm in Theorem \ref{thm:unbndmdisp}. One can further improve this by noting that 
\begin{equation}\label{eq:supCD} \sup_{x_1,y } \big( |C_{n,y,x_1,t}|  +  |D_{n,y,x_1,t}| \big) \les \mu_n^{\frac12-\frac\gamma2}\la t\sqrt{\mu_n}\ra^{-1/2}\les \mu_n^{\frac12-\frac\gamma2}\la t \ra^{-1/2}, 
\end{equation}
which, in turn, implies that  (see \eqref{eq:implicit})
 $$
 \|e^{-itD}\la D\ra^{-\gamma} (I-P_0)f\|_{  L^\infty_{x_1} L^2_{x_2}} \les \la t\ra^{-1/2}  
 \Big\|\int_\R   \mu_n^{\frac12-\frac\gamma2} |f_n(y)|   dy\Big\|_{\ell^2_n}=   \la t\ra^{-1/2}  
 \big\|  \mu_n^{\frac12-\frac\gamma2} \|f_n(y)\|_{L^1_y} \big\|_{\ell^2_n}.
 $$
Suppose that $m(x_2)\approx \pm |x_2|^\beta$ as $x_2\to\pm\infty $ for some $\beta>0$, and that consequently, the  potentials associated to $H_{\pm}$ have the asymptotic growth rate $V_{\pm}(x_2)\approx |x_2|^{2\beta}$ as $x_2\to\pm\infty$. The one-dimensional Weyl law \cite{Tischmarsh_1959} then states that the eigenvalues $\mu_n$ in this case grow asymptotically like
$$\mu_n\approx n^{\frac{2\beta}{\beta+1}}.$$
In particular, $\mu_n^{\frac12-\frac\gamma2}\in\ell^2_n$ if $\gamma>\frac32+\frac1{2\beta}$.
Whence, in this case,  the bound \eqref{eq:disp_unb} in Theorem \ref{thm:unbndmdisp} can be improved to 
\begin{equation}\label{eq:supnL1}
 \|e^{-itD}\la D\ra^{-\gamma} (I-P_0)f\|_{  L^\infty_{x_1} L^2_{x_2}} \les \la t\ra^{-1/2}  \|f_n(y)\|_{\ell^\infty_nL^1_y}. 
\end{equation}
Finally, for $|t|>1$, we can ``pull'' another factor of $\mu_n^{-1/4}$ by noting that $\langle t\mu_n^{1/2}\rangle^{-1/2}\lesssim t^{-1/2}\mu_n^{-1/4}$, and so the bound \eqref{eq:supCD} gives a total factor of $|t|^{-\frac12} \mu_n^{\frac14-\frac\gamma2}$. Note that  $ \mu_n^{\frac14-\frac\gamma2}\in\ell^2_n$ for $\gamma>1+\frac1{2\beta}$. Therefore, for $|t|>1$, \eqref{eq:supnL1} holds provided that $\gamma>\max(\frac32,1+\frac1{2\beta})$.  
\end{rmk}

\section{Klein-Gordon bounds}\label{sec:KG}
We have the following well-known bound for the free   Klein-Gordon propagator. We   provide a proof below for completeness.
\begin{lemma}\label{lem:KGdisp} Let $d=1,2$. 
The Fourier transform of $f_t(k)\equiv e^{-it\nu(k)} \nu(k)^{-\gamma }$ on $\R^d $ is bounded by $\la t\ra^{-d/2}$ provided that $\gamma >1+\frac{d}2$.  
\end{lemma}
\begin{proof} Let $\chi_0$ be a smooth cut-off for $[-1,1]$ and $\chi $ for $[1/2,2]$ so that $\chi_0(k)+\sum_{j=1}^\infty\chi_j(k) \equiv 1 $ on 
$\R$, where $\chi_j(k)\equiv\chi(2^{-j}|k|)$, $j=1,2,3,..$ .

When $d=1$, we have 
$$
\widehat{f_t}(\eta)= \int_\R e^{-it\la k\ra}e^{-ik\eta} \la k\ra^{-\gamma} \,  dk \, .   
$$
First note that the integrand is in $L^1$ as $\gamma>3/2$. Hence we have the uniform bound $|\widehat{f_t}(\eta)|\les 1$. 
It remains to consider the case $|t|\gg 1$. We write
$$
\widehat{f_t}(\eta)= \sum_{j=0}^\infty \int_\R e^{-it\la k\ra}e^{-ik\eta} \chi_j(k) \la k\ra^{-\gamma} dk \, .  
$$
Let $\phi (k)\equiv-t\la k\ra-k\eta$. Note that $|\phi^{\prime\prime}(k)|=|t|\la k\ra^{-3} \approx |t|2^{-3j}$ on the support of $\chi_j$. Therefore, by Van der Corput's lemma, the $j$-th integral is bounded by 
$$
|t|^{-1/2}2^{3j/2} \|\partial_k [\chi_j(k) \la k\ra^{-\gamma}]\|_{L^1(\R)} \les |t|^{-1/2} 2^{(\frac32-\gamma)j} \, .
$$   
Summing over $j$ yields the claim as $\gamma>3/2$. 

When $d=2$, since $f_t$ is a radial function we have (ignoring constants)  
$$
\widehat{f_t}(\eta)= \int_0^\infty e^{-i t\la r\ra } J_0(r \rho) \frac{r}{\la r\ra^\gamma}  \, ,dr \, ,  
$$
where $\rho\equiv|\eta|$, and $J_0$ is the $0$-th order Bessel function, which admits the well-known representation (obtained by the asymptotic expansion of $J_0$, see e.g., \cite[Section 2]{erdougan2018dispersive}),
$$   J_0(r\rho) = e^{ir\rho}\omega_{+}(r\rho) + e^{-ir\rho}\omega_{-}(r\rho) \, .
$$ 
It thus suffices to consider 
$$
 I\equiv \int_0^\infty e^{i \phi(r)}  \omega_{\pm} (r \rho) \frac{r}{\la r\ra^\gamma} \,  dr \, , 
$$
where as before $\phi(r)\equiv  t \sqrt{1+r^2}\pm  r\rho$ and, again from the standard asymptotics of the Bessel function, 
\begin{equation}\label{eq:omegabounds}|\partial_r^N\omega_{\pm}(r\rho)|\les \frac1{r^{N}\la r\rho\ra^{1/2}}, \qquad N=0,1,2, \cdots  \, .
\end{equation}
Once again, the integrand is in $L^1$ (since $\gamma>2$). Therefore, it suffices to consider the case $|t|\gg 1$.
  
First consider the case $\rho\gtrsim |t|$. We write
$$I=\sum_{j=0}^\infty I_j\equiv \sum_{j=0}^\infty  \int_0^\infty e^{i \phi(r)} \chi_j(r) \omega_{\pm} (r \rho) \frac{r}{\la r\ra^\gamma} \, dr  \, .$$
Note that, on the support of $\chi_j$, $j\geq 0$, we have $|\phi^{\prime\prime}(r)|=|t|\la r\ra^{-3}\approx |t| 2^{-3j}$. Therefore, by Van der Corput's lemma, we have 
$$
|I_j|\les |t|^{-1/2}2^{3j/2} \|\partial_r [\chi_j(r) r \omega_{\pm} (r \rho) \la r\ra^{-\gamma}]\|_{L^1(\R_r)} \les |t|^{-1/2} \rho^{-1/2} 2^{(2-\gamma)j} \les |t|^{-1}2^{(2-\gamma)j}.
$$   
Here, we used  the bound $\rho\gtrsim|t| $ and  \eqref{eq:omegabounds}  on  $\omega_{\pm}$. Summing over $j$ yields the claim as $\gamma>2$. 

Now, consider the case $\rho\ll |t|$ and $1\ll |t|$.  Let  $\chi_0$ and $\chi_j$ be as above, and let  $\widetilde\chi=1-\chi_0$. On the support of $\widetilde\chi $,  we have  
\begin{equation}\label{eq:phasebounds}
| \phi^\prime(r)|=\big|\tfrac{tr}{ \sqrt{1+r^2}}\pm \rho\big| \gtrsim |t| \, , \qquad  |\phi^{\prime\prime}(r)| \sim |t| \la r\ra^{-3} \, .
\end{equation}
Writing $e^{i\phi}=\frac{1}{i\phi^\prime}\partial_r e^{i\phi}$ to integrate by parts,  we bound the contribution of $\widetilde\chi$ to $I$   by
$$
\Big|\int_0^\infty e^{i \phi(r)} \widetilde \chi(r) \omega_{\pm} (r \rho) \frac{r}{\la r\ra^\gamma} \,dr \Big| \les   \int_0^\infty \Big| \partial_r \Big(\frac{1}{\phi^\prime(r)} \omega (r \rho) \widetilde\chi(r) \frac{ r}{\la r\ra^{\gamma }}\Big)\Big| dr.
$$
 Using the bounds  \eqref{eq:omegabounds} and \eqref{eq:phasebounds} on $\omega$  and $\phi$, respectively, we have the  bound 
 $$
 \les \int_0^\infty \Big(\frac1{\la r\ra^\gamma |t|}+\frac{ r }{\la r\ra^{3+\gamma}|t| }\Big) dr\les \frac1{|t|} \, ,  $$  
 for which even $\gamma>1$ suffices.
To complete the proof of a decay bound $|t|^{-1}$ when $\rho \ll |t|, $ $1\ll |t| $, it remains to consider the contribution of $\chi_0$, i.e.,  to bound the integral
$$ I_0 \equiv \int_0^\infty e^{i \phi(r)} \chi_0(r) \omega_{\pm} (r \rho) \frac{r}{\la r\ra^\gamma} \, dr  \, . $$
By inserting $1=\chi_0 (r|t|^{1/2})+\widetilde{\chi}(r|t|^{1/2})$, we have 
 \begin{equation} \label{eq:Idecomp}
 |I_0|\les  \int_0^\infty  \chi_0(r|t|^{1/2}) r \,dr + \Big|\int_0^\infty   e^{i \phi(r)} \chi_0(r) \widetilde\chi(r|t|^{1/2}) \omega_{\pm} (r \rho) \frac{ r}{\la r\ra^{\gamma }} \, dr\Big| \, .
 \end{equation}
 The first integral $\sim \int_0^{t^{-1/2}} r \, dr \sim |t|^{-1}$.  

For the second term in the bound for $I_0$, let $r_0$ be the critical point of $\phi$ on the support of $\chi_0$: $\frac{r_0}{\la r_0\ra  } = \rho/|t|$. Note that $r_0\sim \rho/|t|$. We further divide the integral into two pieces, near that critical point and away from it:
\begin{multline*}
\int_0^\infty e^{i \phi(r)}  \widetilde\chi(rt^{1/2}) \chi_0(r) \chi_0((r-r_0)|t|^{1/2})\omega_{\pm} (r \rho) \frac{r}{\la r\ra^{\gamma }} dr \\+\int_0^\infty e^{i \phi(r)}  \widetilde\chi(r|t|^{1/2}) \chi_0(r) \widetilde\chi((r-r_0)t^{1/2})\omega_{\pm} (r \rho) \frac{r}{\la r\ra^{\gamma }} dr.
\end{multline*}
The first integral is bounded by
$|t|^{-1/2}\min(r_0+|t|^{-1/2},\tfrac{(r_0+|t|^{-1/2})^{1/2}}{\rho^{1/2}})$: this is because we integrate on an interval of length $\les |t|^{-1/2}$, and bound the integrand by first using $|\omega_{\pm}(r\rho)|\les \min(1, (r\rho)^{-1/2})$ from \eqref{eq:omegabounds}. That upper bound is in turn
 bounded by $|t|^{-1}$ if $r_0\les |t|^{-1/2}$. If $ r_0\gg |t|^{-1/2}$, we have  $|t|^{-1/2} \tfrac{(r_0+|t|^{-1/2})^{1/2}}{\rho^{1/2}} \sim |t|^{-1/2} \tfrac{(\rho/|t|)^{1/2}}{\rho^{1/2}}\sim  |t|^{-1}$ as well.   Integrating by parts twice, we bound the second integral, yielding 
\begin{equation}\label{eq:awayfromr0_kg2d}
\int_0^\infty \Big|\partial_r \Big(\frac{1}{\phi^\prime(r)} \partial_r \Big(\frac{1}{\phi^\prime(r)}\widetilde\chi(r|t|^{1/2}) \chi_0(r) \widetilde\chi((r-r_0)|t|^{1/2}) \omega _{\pm}(r \rho) \frac{ r}{\la r\ra^{\gamma }}\Big)\Big)\Big| dr\, .
\end{equation}
Note that, as $r\to r_0$,
$$
|\phi^\prime(r)|=|\phi^\prime(r)-\phi^\prime(r_0)| \sim |r-r_0| |t| \, , \qquad  |\phi^{\prime\prime}(r)|\sim |t|\, , \qquad |\phi^{(3)}(r)|\sim r|t| \, .
$$
Write  $\widetilde{\chi}_t (r) \equiv \widetilde\chi_1(r|t|^{1/2}) \widetilde\chi_1((r-r_0)|t|^{1/2}) $, where $\widetilde{\chi}_1$ has a slightly larger support so that the support of the derivatives of $\widetilde{\chi}$ are included. Then, the integral in \eqref{eq:awayfromr0_kg2d} is bounded by 
\begin{align*}\int_0^1 \frac{1 }{ r|r-r_0|^2|t|^2}\widetilde{\chi}_t (r) dr &+\int_0^1 \frac{1 }{  |r-r_0|^3|t|^2} \widetilde{\chi}_t (r) \, dr \\
&+ \int_0^1 \frac{r  }{  |r-r_0|^4|t|^2\la r\rho\ra^{1/2}} 
\widetilde{\chi}_t (r) dr \, .
\end{align*}
 The first integral is $O(|t|^{-1})$ by first using the bound $r^{-1}\les |t|^{\frac12}$ on the support of $\widetilde\chi_t$, and then integrating on the set $|r-r_0|\gtrsim |t|^{-\frac12}$. The second integral is bounded by $|t|^{-1}$ by direct integration on the set $|r-r_0|\gtrsim |t|^{-\frac12}$.  

 For the third integral, we consider the cases $r\sim r_0$ and $r\les |r-r_0|$ separately on the support of $\widetilde\chi((r-r_0)|t|^{1/2})$.   When $r\sim r_0$ we have that, since $r_0\sim \rho/|t|$ and $|r-r_0|\gtrsim|t|^{-1/2}$  
$$  \frac{r  }{  |r-r_0|^4 \la r\rho\ra^{1/2}} \les \frac{r_0}{ |r-r_0|^4   r_0^{1/2}\rho^{1/2}} \les  \frac{|t|^{-1/2}}{|r-r_0|^4 } \les  \frac{1}{|r-r_0|^3 }. $$
Similarly,  when $r\les |r-r_0|$, we simply use the fact that $\langle \rho r\rangle ^{-1/2}\lesssim 1$ to get  
$$
 \frac{r  }{  |r-r_0|^4 \la r\rho\ra^{1/2}}\les \frac{1}{  |r-r_0|^3  } \, .
$$
Therefore, in both cases the integral is bounded by the second integral, which contributes  $O(|t|^{-1})$. 
\end{proof}

\appendix

\bibliographystyle{abbrv}
\bibliography{mybib}

\end{document}